\documentclass[11pt]{amsart}
\usepackage{amsmath, amsthm, amssymb, color,xcolor}
\usepackage[ pdfborderstyle={/S/U/W 1}]{hyperref}
\definecolor{MyBlue}{HTML}{210cac}
\definecolor{MyCiteColor}{HTML}{0099FF}
\definecolor{MyRed}{HTML}{3E186A}
\definecolor{Red2}{HTML}{FF6E00}
\usepackage{tikz-cd}
\hypersetup{%
  pdfborderstyle={/S/U/W 1},
  colorlinks=true,
  linkcolor=MyBlue,
  citecolor=MyCiteColor,
  urlcolor=MyRed,
  anchorcolor=MyBlue,
  linkbordercolor=MyRed,
}
	\definecolor{oran}{HTML}{0A7F5E}
\usepackage[shortlabels]{enumitem}
\usepackage[final]{pdfpages}
\setboolean{@twoside}{false}
\usepackage{caption}
\usepackage{subcaption}
\usepackage{mathrsfs,upgreek}
\usepackage{scalefnt}
\usepackage{verbatim}
\usepackage{wrapfig}
\calclayout

\makeatletter
\newcommand{\eqnum}{\refstepcounter{equation}\textup{\tagform@{\theequation}}}

\makeatother

\newtheorem{theorem}{Theorem}
\numberwithin{theorem}{section}
\newtheorem{proposition}[theorem]{Proposition}
\newtheorem{lemma}[theorem]{Lemma}
\newtheorem{corollary}[theorem]{Corollary}

\theoremstyle{definition}
\newtheorem{definition}[theorem]{Definition}

\newtheorem{remark}[theorem]{Remark}
\newtheorem{example}[theorem]{Example}

\newcommand{\QQ}{\mathbb{Q}}
\newcommand{\pp}{\mathbb{P}}

\newcommand{\CC}{\mathbb{C}}
\newcommand{\kk}{\mathfrak{K}}
\newcommand{\ZZ}{\mathbb{Z}}
\newcommand{\sG}{\mathcal{G}}

\newcommand\rank{\mathrm{rank~}}
\newcommand\Span{\mathrm{span}}
 \date{}
 \usepackage{booktabs}
\usepackage{colortbl}
 \definecolor{Ftitle}{RGB}{11,46,108}
\definecolor{line}{RGB}{87,39,117}
\colorlet{tableheadcolor}{Ftitle!25} 
\colorlet{tablerowcolor}{gray!10} 

\newcommand{\conj}{{\rm conj}}
\newcommand{\Null}{{\rm null}}

\newcommand\Var{{\mathcal{V}}}
\usepackage{multirow}

\title{Probabilistic Saturations and Alt's Problem}
\author{Jonathan D. Hauenstein}
\author{Martin Helmer}

\begin{document} 

\begin{abstract} \noindent
Alt's problem, formulated in 1923, is to count the number of
four-bar linkages whose coupler curve interpolates nine general
points in the plane.  This problem can be phrased 
as counting the number of solutions to a system of polynomial equations
which was first solved numerically using homotopy continuation 
by Wampler, Morgan, and Sommese in 1992.  Since
there is still not a proof that all solutions were obtained, 
we consider upper bounds for Alt's problem
by counting the number of solutions outside of the base locus
to a system arising as the general linear combination of polynomials.
In particular, we derive effective symbolic and numeric methods 
for studying such systems using probabilistic saturations
that can be employed using both finite fields and floating-point computations.
We give bounds on the size of finite field required to achieve a 
desired level of certainty.  These methods can also be applied 
to many other problems where similar systems arise such as 
computing the volumes of Newton-Okounkov bodies and 
computing intersection theoretic invariants 
including Euler characteristics, Chern classes, and Segre classes. 
\end{abstract}
\maketitle

\section{Introduction}

In 1923, Alt \cite{Alt} realized that finitely many four-bar planar linkages
(see Figure~\ref{fig:4bar})
can be constructed whose coupler curve interpolates nine general points 
in the plane.  Since he could not determine the exact number himself,
he left this as an open problem which was solved numerically
using homotopy continuation nearly 70 years later
by Wampler, Morgan, and Sommese \cite{WMS}.
In particular, they showed that there are 1442 distinct four-bar coupler
curves which pass through nine general points which, together with
Roberts cognates \cite{R}, yields 4326 distinct four-bar linkages.
Although this computation has been repeatedly confirmed
using various homotopy continuation methods, 
e.g., \cite{CCPC,Alts1,MhomTrace,Regeneration,FRG,ConstrainedHomotopy},
these numerical computations 
do not preclude the existence of additional solutions.  
In fact, one of the distinct four-bar linkages
was missed by the homotopy continuation solver in \cite{WMS} 
but was reconstructed using the cognate formula.
Since a sharp upper bound has not yet been established, we
aim to derive such an upper bound by considering a generalization
of Alt's problem which arises by counting the number of solutions
outside of the base locus to a system of polynomials
arising from a general linear combination of given polynomials.

Although formulating an upper bound on Alt's problem in this
fashion is new, the idea of constructing polynomial systems
from randomly generated polynomials which have finitely-many solutions 
is not new.  For example, the volume of 
Newton-Okounkov bodies \cite{KK12,LM09,Okunkov96}
(which generalize the volume of Newton polytopes \cite{BKK} 
in the monomial case) correspond with 
the number of solutions to such systems.  
Homotopies utilizing this structure have also been proposed \cite{BeyondPolyhedral}
and comparisons between homotopy continuation and
modular Gr\"obner basis methods have been performed \cite{Comparison}.
Moreover, aspects of this problem have also been considered 
in the context of computing Euler characteristics \cite{H17,MB12} 
and Segre classes~\cite{HH18,H16}. 
In addition to formulating an upper bound on Alt's problem,
the other novelty of our approach is to 
provide various probabilistic features of an
effective algorithmic solution via modular Gr\"obner basis computations.
We incorporate ideas used in the Gr\"obner trace algorithm \cite{GBTrace} and other modular algorithms for Gr\"obner basis computation such as~\cite{ModGB}.

The rest of the paper is structured as follows. 
Section~\ref{section:background} formulates the problem 
and provides necessary background.  
Section~\ref{sec:HilbertFunctions} compares Hilbert functions 
computed using numerical methods and symbolic methods.
Section~\ref{sec:ProbRandomIsGeneral} analyses the probability
that randomly selected constants are general. Section~\ref{sec:Computations} 
describes the results of some computational experiments 
and the paper ends with a summary of our results in the context of Alt's problem in Section \ref{Section:Conclusion}.

\section{Formulation}\label{section:background}

For a system of polynomials in $\CC[x_1,\dots,x_n]$, say
\begin{equation}
f(x)=\begin{bmatrix}
f_1(x)\\
f_2(x)\\
\vdots\\
f_r(x)
\end{bmatrix},
\end{equation} 
define
$$\Var(f) = \{x\in \CC^n~|~f(x) = 0\}.$$
For $X = \Var(f)$ and an integer $i=0,\dots,n-1$, we seek to compute
\begin{equation}
g_i(X, \CC^n)=\deg \left(\Var(\Theta\cdot x-{\mathbf{1}}, \Lambda \cdot f)\setminus X \right)\label{eq:g_is}
\end{equation}
where $\Theta\in \CC^{i\times n}$ and $\Lambda\in \CC^{(n-i)\times r}$ are general, and $\mathbf{1}$ is the length $i$ vector where each entry
is $1$. That is, we aim to count the number of solutions to a system 
of $i$ general affine linear polynomials and $n-i$ general linear combinations
of $f_1,\dots,f_r$ outside of the base locus~$X$.

One approach for enforcing the solutions be 
outside of the base locus $X$ is via the Rabinowitz trick.
That is, one introduces a new variable $T$ and, for $\mu\in\CC^r$ general,
we have
$$g_i(X,\CC^n) = \deg(\Var(\Theta\cdot x-1, \Lambda\cdot f, 1 - (\mu\cdot f)T)).$$

\begin{example}
To illustrate, consider $f(x) = [x_1,x_2,x_1x_2^2,x_1^3x_2^2]^T$
consisting of $r = 4$ polynomials in $n=2$ variables.
Clearly, $\Var(f) = \{(0,0)\}$.  Letting $\Box$ represent a general number, 
$$g_1(X,\CC^2) = \deg(\Var(\Box x_1 + \Box x_2 - 1, \Box x_1 + \Box x_2 + \Box x_1x_2^2 + \Box x_1^3 x_2^2)\setminus X) = 5$$
since there are $5$ points of intersection,
all of which are away from the origin,
between a general affine line and an irreducible quintic curve.
Additionally, 
$$g_0(X,\CC^2) = \deg(\Var(\Box x_1 + \Box x_2 + \Box x_1x_2^2 + \Box x_1^3 x_2^2, \Box x_1 + \Box x_2 + \Box x_1x_2^2 + \Box x_1^3 x_2^2)\setminus X) = 6$$
since the two irreducible quintic curves intersect in $7$
points, one of which is the origin.
Since~$f$ consists of monomials, 
these values can
be computed via mixed volume computations, e.g.,~\cite{BKK}.\label{example:firstEx}
\end{example}

\subsection{Relation to Alt's problem} \label{subSection:AltsIntro}

A four-bar linkage is the simplest moveable planar closed-chain linkage,
one of which is shown in Figure~\ref{fig:4bar}
together with part of the coupler curve that it traces out.
Without loss of generality, we can assign
the origin $0$ to be a point on the coupler curve 
and describe a four-bar linkage using $4$ points
in the plane: $A$, $B$, $X$, and $Y$.  
The points~$A$ and $B$ correspond with the two fixed pivots
while the segment $XY$ is the so-called floating link.
Since the $4$ planar points constitute $n = 8$
design parameters, Alt realized that one
can additionally specify $8$ general points 
resulting in finitely many four-bar linkages
whose coupler curve passes through all nine specified points,
i.e., the origin and $8$ others.

\begin{figure}[!t]
\begin{center}
\begin{picture}(150,150)
\put(0,0){\includegraphics[scale=0.35]{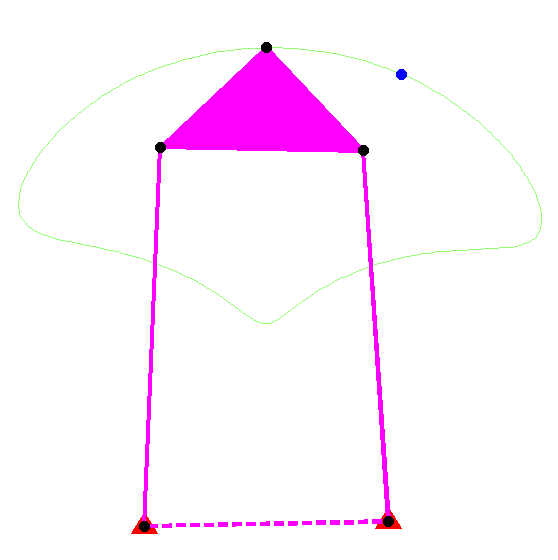}}
\put(100,105){$Y$}
\put(106,9){$B$}
\put(30,105){$X$}
\put(25,8){$A$}
\put(68,140){$0$}
\put(105,132){$P$}
\end{picture}
\end{center}
\caption{A four-bar linkage with corresponding coupler curve}\label{fig:4bar}
\end{figure}
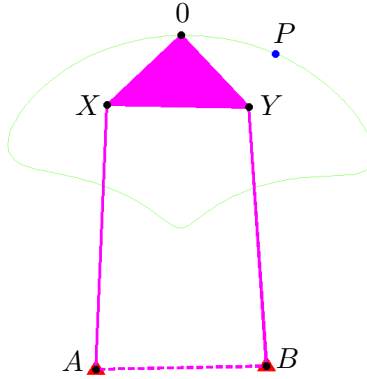

We proceed with a formulation inspired by ~\cite{RF63,WMS}.
Utilizing isotropic coordinates, 
we write $A = (a,\bar{a})$, $B = (b,\bar{b})$,
$X = (x,\bar{x})$, and $Y = (y,\bar{y})$
so that the eight design parameters
are $a,\bar{a},b,\bar{b},x,\bar{x},y,\bar{y}$.  
A point $P$ on the coupler
curve will be denoted $P = (p,\bar{p})$.
Isotropic coordinates yield
a natural action applied to polynomials, denoted $\conj()$, 
that simply swaps~$z$ and $\bar{z}$.
In particular, given design parameters 
$(a,\bar{a},b,\bar{b},x,\bar{x},y,\bar{y})$,
the coupler curve traced out by the corresponding four-bar linkage
is $\Var(G)\subset\CC^2$ where 
$$G = \sum_{j=1}^{15} c_j(p,\bar{p})\cdot f_j(a,\bar{a},b,\bar{b},x,\bar{x},y,\bar{y})$$
with
$$
\begin{array}{lllll}
c_{1} = p^3 \bar{p}^3, & c_{2} = p^3 \bar{p}^2, & 
c_{3} = \conj(c_{2}), & c_{4} = p^3 \bar{p}, & c_{5} = \conj(c_{4}), \\[0.04in]
c_{6} = p^3, & c_{7} = \conj(c_{6}), & c_{8} = p^2\bar{p}^2, & c_{9} = p^2\bar{p}, & c_{10} = \conj(c_{9}), \\[0.04in]
c_{11} = p^2, & c_{12} = \conj(c_{11}), & c_{13} = p\bar{p}, & c_{14} = p, & c_{15} = \conj(c_{14}), \\
\end{array}
$$
and 
$$\begin{array}{l}
f_1 = (x-y)(\bar{y}-\bar{x}), \\
f_2 = (x-y)(\bar{a} \bar{x}-2 \bar{a} \bar{y}+2 \bar{b} \bar{x}-\bar{b} \bar{y}), \\
f_3 = \conj(f_2), \\
f_4 = (x-y)(\bar{a}^2 \bar{y}-2 \bar{a} \bar{b} \bar{x}+2 \bar{a} \bar{b} \bar{y}-\bar{b}^2 \bar{x}),\\
f_5 = \conj(f_4), \\
\end{array}$$
$$\begin{array}{l}
f_6 = \bar{a} \bar{b} (x-y) (\bar{b} \bar{x}-\bar{a}\bar{y}),\\
f_7 = \conj(f_6),\\
f_8 = x^2 \bar{y} (\bar{a}-\bar{y})
+ \bar{x}^2 y (a-y) 
+ x\bar{x}(2y\bar{x}+(\bar{a}-2\bar{b})y+(a-2b)\bar{y}
-a\bar{a}-2a\bar{b}-2\bar{a}b-2b\bar{b}) \\
\,\,\,\,\,\,\,\,\,\,\,\, +~x(b\bar{y}^2 + y\bar{y}(\bar{b}-2\bar{a}) + \bar{y}(a\bar{a}+a\bar{b}+4\bar{a}b + b\bar{b}))- y\bar{y}(2a\bar{a}+2a\bar{b}+2\bar{a}b+b\bar{b})  \\
\,\,\,\,\,\,\,\,\,\,\,\, +~\bar{x}(
\bar{b}y^2 + y\bar{y}(b-2a) + y(a\bar{a}+\bar{a}b+4a\bar{b}+b\bar{b})), \\
f_9 = \bar{a} x^2 \bar{y} (2 \bar{y} - \bar{a} - \bar{b})
+ 2 \bar{b} \bar{x}^2 y (y - a) 
- \bar{a} x \bar{y}(a \bar{b} + 2\bar{a}(b-y)+2b(\bar{b}+\bar{y}))
\\
\,\,\,\,\,\,\,\,\,\,\,\, 
+ x\bar{x}(
\bar{b}(2\bar{b}y + 2 a \bar{a} + 2 \bar{a} b + a \bar{b})
+\bar{y}(\bar{a}+\bar{b})(2b-a-2y)) 
+\bar{a} y \bar{y} (2 a \bar{b}+ \bar{a} b +2 b \bar{b})
\\
\,\,\,\,\,\,\,\,\,\,\,\, 
+~\bar{x} y 
((2a\bar{y}-2a\bar{b}-b\bar{y}-\bar{b}y)(\bar{a}+\bar{b}) -\bar{a} b \bar{b}), \\
f_{10} = \conj(f_9), \\
f_{11} =  \bar{a}^2 x^2 \bar{y} (\bar{b}-\bar{y})
+ \bar{b}^2 \bar{x}^2 y(a-y)
- \bar{a} \bar{b} x \bar{x}(a (\bar{b}-\bar{y}) + 2\bar{y} (b-y) + \bar{b} y)
+ \bar{a}^2 x \bar{y}(b\bar{b}+ b\bar{y}-\bar{b} y) \\
\,\,\,\,\,\,\,\,\,\,\,\, +~ \bar{a}\bar{b} \bar{x} y (\bar{b}(a+y) + \bar{y} (b - 2a)) - \bar{a}^2 b \bar{b} y \bar{y}, \\
f_{12} = \conj(f_{11}), \\
f_{13} = 
\bar{a} x^2 \bar{y} (\bar{a} (2 b - y) + b(\bar{b} - 3\bar{y}) + \bar{b} y)
+ a \bar{x}^2 y (a (2\bar{b} - \bar{y}) + \bar{b}(b-3y) + b\bar{y})
\\
\,\,\,\,\,\,\,\,\,\,\,\, +~ 
x \bar{x} (\bar{b} y^2 (\bar{a}-\bar{b})
+ 3 y \bar{y} (a\bar{b}+\bar{a}b)
+ b \bar{y}^2 (a - b)
- \bar{b} y (\bar{a} b + 2 a \bar{b})
- b \bar{y} (a \bar{b} + 2 \bar{a} b)
- 2 a \bar{a} b \bar{b}) \\
\,\,\,\,\,\,\,\,\,\,\,\, +~
\bar{a} x \bar{y}(a(b\bar{b} + b \bar{y} - \bar{b} y) + b (\bar{a} (b - 2 y) + 2 b \bar{y})) + 
a \bar{x} y (
\bar{a} (b \bar{b} - b \bar{y} + \bar{b} y) \\
\,\,\,\,\,\,\,\,\,\,\,\, +~\bar{b} (a (\bar{b} - 2 \bar{y}) + 2 \bar{b} y)) - 2 a \bar{a} b \bar{b} y \bar{y},
\\
f_{14} = (\bar{a} b x \bar{y} - a \bar{b} \bar{x} y)((\bar{a} - \bar{x})(\bar{b} y - b \bar{y})+ (\bar{b} - \bar{y}) (a \bar{x} - \bar{a} x)), \\
f_{15} = \conj(f_{14}).
\end{array}
$$
In fact, the map from design parameters to coupler curves
is generically a $6$-to-$1$ map.  First,
there is a trivial order $2$ action of relabeling, namely 
the design parameters
$$(a,\bar{a},b,\bar{b},x,\bar{x},y,\bar{y})
\hbox{\,\,\,\,\,and\,\,\,\,\,} (b,\bar{b},a,\bar{a},y,\bar{y},x,\bar{x})$$
yield the same coupler curve and actually describe the same mechanism.
Additionally, there is an order $3$ action of Roberts cognates,
e.g., see \cite[Eq.~17]{WMS}, which yield 
distinct mechanisms that correspond to the same coupler curve.

Returning to Alt's problem, suppose
that $(p_i,\bar{p}_i)$ for $i = 1,\dots,8$ are general.  
Thus, one aims to find the design parameters $(a,\bar{a},b,\bar{b},x,\bar{x},y,\bar{y})$ such that, for $i = 1,\dots,8$, 
$(p_i,\bar{p}_i)$ lies on its coupler curve.
Hence, for $i = 1,\dots,8$ and
$$G_i = \sum_{j=1}^{15} c_j(p_i,\bar{p}_i)\cdot f_j(a,\bar{a},b,\bar{b},x,\bar{x},y,\bar{y}),$$
Alt's problem of counting the number of distinct mechanisms
is equal to one-half of the number of isolated points in $\Var(G_1,\dots,G_8)$
while the number of distinct coupler curves is equal
to one-sixth of the number of isolated points in $\Var(G_1,\dots,G_8)$.
The homotopy continuation computation 
first described in \cite{WMS} provides that
the number of isolated points in $\Var(G_1,\dots,G_8)$~is (at least)~$8652$.

We obtain an upper bound on the number of isolated points
in $\Var(G_1,\dots,G_8)$ by replacing each coefficient
$c_j(p_i,\bar{p}_i)$ with an independent parameter, say $c_{ij}$.
The number of isolated solutions to the resulting
system is exactly $g_0(X,\CC^n)$ from \eqref{eq:g_is}
where $X = \Var(f_1,\dots,f_{15})$.
In fact, it is easy to verify that~$X$ is the union of the following $7$ linear spaces:
\begin{equation}\label{eq:AltBaseLocus}
{\small
\begin{array}{llll}
\Var(x,y), & \Var(x - y, x - b, a - b), & 
\Var(x-y, a-b, \bar{x}-\bar{a}, \bar{y}-\bar{b}),  & \Var(x-y,\bar{x}-\bar{y},a-b,\bar{a}-\bar{b}), \\
\Var(\bar{x},\bar{y}), & \Var(\bar{x}-\bar{y}, \bar{x}-\bar{b}, \bar{a}-\bar{b}), & \Var(\bar{x}-\bar{y}, \bar{a}-\bar{b}, x-a, y-b).
\end{array}
}
\end{equation}
Note that each of these correspond to degenerate linkages that
are not of interest.  Therefore, $g_0(X,\CC^8)/2$ 
is an upper bound on the number of distinct 
four-bar linkages whose coupler curve interpolates~$9$ general points.  Similarly, $g_0(X,\CC^8)/6$ 
is an upper bound on the number of distinct
coupler curves which interpolate $9$ general points.
In addition, one naturally has a coefficient-parameter
homotopy \cite{CoeffParam} for computing
a superset of the isolated points of~$\Var(G_1,\dots,G_8)$
by simply deforming the coefficients 
$c_{ij}$ to $c_j(p_i,\bar{p}_i)$.

For $i = 1,\dots,7$, the number $g_i(X,\CC^8)$
is an upper bound on the degree of
the set of four-bar linkages whose coupler curve interpolates
$9-i$ general points, a variety that has dimension~$i$.
The degrees of these problems 
were first reported in \cite[Table~1]{Alts1}
and are displayed in Table~\ref{table:Degrees}. Similar to the results of \cite{WMS} the computations in \cite{Alts1} do not prove that the degrees can be no larger than the values in Table~\ref{table:Degrees} and we will use modular methods to provide a (probabilistic) confirmation of these values.

\subsection{Relation to projective degrees and characteristic classes}\label{subsec:ProjDegCharClass}
For $X = \Var(f)$, we may also view each $g_i(X,\CC^n)$ as the degree of the pullbacks of general linear spaces under a certain rational map. Consider the rational map \begin{equation}
\begin{aligned}
\rho \colon& \CC^n \dashrightarrow \CC^r\\
&x\mapsto f(x)
\end{aligned}
\end{equation} 
and let $\mathcal{L}_i\subset \CC^r$ be a general linear space
passing through the origin
of dimension $r-(n-i)$ in $\CC^r$, i.e.,~having codimension $n-i$ in $\CC^r$.
Let $B_{n-i} \in \CC^{(n-i)\times r}$ be a matrix
such that $\mathcal{L}_i = \Null(B_{n-i})$.
Then, 
$$
g_i(X,\CC^n)=\deg(\rho^{-1}(\mathcal{L}_i)\setminus X) = \deg(\Var(B_{n-i}\cdot f)\setminus X). 
$$
The number $g_0(X,\CC^n)$ is also 
the volume of the Newton-Okounkov body \cite{KK12,LM09,Okunkov96}
corresponding to $f$.

More generally, if $X=\Var(F)\subset Y$ are subschemes 
inside of $Z$, which is a product of affine and projective spaces,
and $F(x)$ is a collection of $r$ polynomials in the coordinate ring of~$Z$, we can consider the map 
\begin{equation}
\begin{aligned}
\rho_X \colon& Y \dashrightarrow \CC^r \\
&x\mapsto F(x).
\end{aligned}
\end{equation} 
In this case, we wish to compute the numbers
\begin{equation}\label{eq:giXY}
g_i(X,Y)=\deg(\rho_X^{-1}(\mathcal{L}_i)\setminus X)=\deg \left(Y\cap \Var(B_{n-i} \cdot F)\setminus X \right).
\end{equation}
If $Z$ is a projective variety, then the polynomials making up $F$ are homogeneous, $\CC^r$ is replaced by $\pp^{r-1}$, a set of polynomials defining $X=\Var(F)$ can be taken to have the same degree (without changing $X\subset Z$), and the numbers $g_i(X,Y)$ are the {\em projective degrees} of $X$ in $Y$ (see \cite{HH18} for the general case as well as \cite[Example~19.4]{Harris} 
and~\cite{H16,H17} for the special case $g_i(X,\pp^{n-1})$). 
Moreover, it is shown in \cite{H16} that when $X$ is 
a subscheme of~$\pp^{n-1}$, the (pushforward of the) Chern-Schwartz-McPherson class in the Chow ring $A^*(\pp^{n-1})$, namely $c_{SM}(X)$, 
and hence the topological Euler characteristic, namely $\chi(X)$, 
are completely determined by appropriate projective degrees. 
Note that this also gives (the pushforward of) the Chern class of the tangent bundle, $c(T_X)\cap [X] \in A^*(\pp^{n-1})$, since this agrees with $c_{SM}(X)$ whenever~$X$ is smooth, i.e., whenever the Chern class is defined.
For more details, see \cite[\S2.2]{Aluffi}. 
Similarly, if $X\subset Y$ are subschemes of a smooth projective toric variety $T$, it is shown in \cite{HH18} that the projective degrees~$g_i(X,Y)$ completely determine the (pushforward of the) Segre class of $X$ in~$Y$, namely~$s(X,Y)$, in the Chow ring $A^*(T)$. 
This Segre class in turn gives rise to many other invariants of potential interest such as Samuel's algebraic multiplicity of~$Y$ along~$X$, 
namely~$e_XY$, e.g., see \cite[\S4.3]{Fulton} and \cite[\S5]{HH18}, 
and polar classes, e.g., see \cite[Ex.~4.4.5]{Fulton}. These Segre classes have also been used to give a Gr\"obner free test of pairwise containment of projective varieties, e.g.,~see~\cite[\S6]{HH18}.

\subsection{Plane conics}\label{subSection:GeomExample}

To illustrate several possible situations, 
we consider the enumerative geometry problem of counting 
the number of plane conics in $\CC^3$ passing
through the origin and meeting~$6$ general lines.  
First, we associate a plane conic $C\subset\CC^3$ passing through 
the origin with $(a,b)\in \pp^4\times\pp^2$ where 
$$C = \Var(a_0 x^2 + a_1 xy + a_2 y^2 + a_3 x + a_4 y, b_0 z + b_1 x + b_2 y ).$$
We parameterize the $6$ lines by $v_i+tw_i$
where $v_i,w_i\in \CC^3$ are general for $i=1,\dots,6$.  
Then, the enumerative geometry problem is to count the number
of isolated points in $\Var(G_1,\dots,G_6)$,
which is $18$, where $G_i = \sum_{j=1}^{14} c_{j}(v_i,w_i)\cdot f_j(a,b)$ with
$$\begin{array}{lllcclll}
c_{1} &=& (v_1 w_3 - v_3 w_1)^2 &&&
c_{2} &=& (v_2 w_3 - v_3 w_2) (v_1 w_3 - v_3 w_1) \\
c_{3} &=& (v_2 w_3 - v_3 w_2)^2 &&&
c_{4} &=& w_3 (v_1 w_3 - v_3 w_1) \\
c_{5} &=& w_3 (v_2 w_3 - v_3 w_2) &&&
c_{6} &=& w_1 (v_1 w_3 - v_3 w_1) \\
c_{7} &=& 2 v_1 w_2 w_3 - v_2 w_1 w_3 - v_3 w_1 w_2 &&&
c_{8} &=& 2 v_2 w_1 w_3 - v_1 w_2 w_3 - v_3 w_1 w_2 \\
c_{9} &=& w_2 (v_2 w_3 - v_3 w_2) &&&
c_{10} &=& (v_3 w_1 - v_1 w_3) (v_1 w_2 - v_2 w_1) \\
c_{11} &=& (v_2 w_3 - v_3 w_2) (v_1 w_2 - v_2 w_1) &&&
c_{12} &=& w_1 (v_2 w_1 - v_1 w_2) \\
c_{13} &=& w_2 (v_2 w_1 - v_1 w_2) &&&
c_{14} &=& (v_1 w_2 - v_2 w_1)^2, \\
\end{array}
$$
and
$$\begin{array}{lllcclll}
f_1 &=& a_0 b_0^2 &&&
f_2 &=& a_1 b_0^2 \\
f_3 &=& a_2 b_0^2 &&&
f_4 &=& a_3 b_0^2 \\
f_5 &=& a_4 b_0^2 &&&
f_6 &=& a_3 b_0 b_1 \\
f_7 &=& a_3 b_0 b_2 &&&
f_8 &=& a_4 b_0 b_1 \\
f_9 &=& a_4 b_0 b_2 &&&
f_{10} &=& b_0 (a_1 b_1 - 2 a_0 b_2) \\
f_{11} &=& b_0 (a_1 b_2 - 2 a_2 b_1) &&&
f_{12} &=& b_1 (a_4 b_1 - a_3 b_2) \\
f_{13} &=& b_2 (a_4 b_1 - a_3 b_2) &&&
f_{14} &=& a_0 b_2^2 + a_2 b_1^2 - a_1 b_1 b_2. \\
\end{array}
$$
For $F = [f_1,\dots,f_{14}]^T$ and $Y = Z = \pp^4\times\pp^2$, 
we have 
$$X = \Var(F) = 
\Var(b_0,
a_4 b_1-a_3 b_2,
a_2 b_1^2-a_1 b_1 b_2+a_0 b_2^2,
a_2 a_3 b_1-a_1 a_3 b_2+a_0 a_4 b_2,
a_2 a_3^2-a_1 a_3 a_4+a_0 a_4^2)$$
which corresponds with planes $\Var(b_1x+b_2y)$ 
that are not of interest.
In fact, $g_0(X,Y) = 18$ providing a sharp upper bound
on the number of isolated points in $\Var(G_1,\dots,G_6)$.

Alternatively, we could work on the affine patch $Y = \CC^4\times\CC^2=\CC^6$ 
defined by $a_0=b_0=1$, i.e., plane cubics defined by
$$\Var(x^2 + a_1 xy + a_2 y^2 + a_3 x + a_4 y, z + b_1 x + b_2 y).$$
Hence, 
\begin{equation}\label{eq:CubicsAffine}
\begin{array}{l} 
F = [1, a_1, a_2, a_3, a_4, a_3 b_1, a_3 b_2, a_4 b_1, a_4 b_2, a_1 b_1 - 2 b_2, \\
\,\,\,\,\,\,\,\,\,\,\,\,\,\,\,\,\,\,\,\,\,\,\,\,\,\,\,\, a_1 b_2 - 2 a_2 b_1,
 b_1 (a_4 b_1 - a_3 b_2), b_2 (a_4 b_1 - a_3 b_2), a_2 b_1^2 - a_1 b_1 b_2 + b_2^2]^T\end{array}
\end{equation}
with $X = \Var(F) = \emptyset$ and $g_0(X,Y) = 18$, which we will verify
in Section~\ref{subSection:HilbertExample}.

Finally, one could work on $Y = \pp^6$, i.e., plane cubics defined by
$$\Var(a_0x^2 + a_1 xy + a_2 y^2 + a_3x + a_4 y, a_0z + b_1 x + b_2 y),$$
with 
$$\begin{array}{l} 
F = [a_0^3, a_0^2 a_1, a_0^2 a_2, a_0^2 a_3, a_0^2 a_4, a_0 a_3 b_1, a_0 a_3 b_2, a_0 a_4 b_1, a_0 a_4 b_2, a_0 (a_1 b_1 - 2 a_0 b_2), \\
\,\,\,\,\,\,\,\,\,\,\,\,\,\,\,\,\,\,\,\,\,\,\,\,\,\,\,\,
a_0(a_1 b_2 - 2 a_2 b_1),
 b_1 (a_4 b_1 - a_3 b_2), b_2 (a_4 b_1 - a_3 b_2), a_2 b_1^2 - a_1 b_1 b_2 + b_2^2]^T\end{array}$$
so that 
$$X = \Var(F) = 
\Var(a_0,a_4b_1-a_3b_2,a_2b_1-a_1b_2,a_2a_3-a_1a_4) \cup \Var(a_0,b_1,b_2)
\cup \Var(a_0,a_3,b_1),$$
which all correspond to degenerate cases that are not of interest,
and $g_0(X,Y) = 18$.

\section{Verification using Hilbert functions}\label{sec:HilbertFunctions}

For homogeneous ideals, modular computations can be utilized
to provide upper bounds on the Hilbert function \cite[Thm.~5.3]{ModGB}.
However, since the ideals of particular interest are 
parameterized ideals which are generically radical and 
zero-dimensional but need not be homogeneous,
this section compares affine Hilbert functions computed
using numerical methods (yielding lower bounds) and
symbolic methods (yielding upper bounds).
In particular, when the upper and lower bounds agree, 
one has computed the generic value of the Hilbert function.

Consider the field $\kk = \QQ(p_1,\dots,p_\ell)$
consisting of rational functions in
the parameters $\mbox{$p=(p_1,\dots,p_\ell)$}$~with rational coefficients.
Let $R = \kk[x_1,\dots,x_n]$ be the ring of polynomials
in the variables $x_1,\dots,x_n$ whose coefficients are in $\kk$.
Fix $p^* = (p_1^*,\dots,p_\ell^*)\in\CC^\ell$.
For a polynomial $h\in R$, let $h_{p^*}\in \CC[x_1,\dots,x_n]$ be the polynomial 
obtained by specializing the parameters~$p$ to~$p^*$ 
whenever all coefficients of $h$ are defined at $p^*$.  
Suppose that $h_1,\dots,h_n \in R$ generate an ideal $I = \langle h_1,\dots,h_n\rangle\subset R$ that is generically radical and zero-dimensional.  
That is, for general values $p^*\in\CC^\ell$, 
$I_{p^*} = \langle {h_{1}}_{p^*},\dots,{h_{n}}_{p^*}\rangle\subset\CC[x_1,\dots,x_n]$
with $\#\Var(I_{p^*}) = \deg I_{p^*} < \infty$.

For $d\geq 0$, let $R_{\leq d}$ and $I_{\leq d}$
denote the vector space over $\kk$ 
of polynomials in $R$ and $I$, respectively,
of degree at most $d$.  In particular, $\dim_{\kk} (R_{\leq d}) = \binom{n+d}{d}$. 
The {\em affine Hilbert function} of $I$ is the function
${\rm HF}_I:\ZZ_{\geq 0}\to \ZZ_{\geq 0}$ defined by
\begin{equation}
{\rm HF}_I(d):= \dim_\kk(\kk[x_1,\dots,x_n]_{\leq d})-\dim_\kk(I_{\leq d})
= \binom{n+d}{d} - \dim_\kk(I_{\leq d}).
\end{equation}
Moreover, for general $p^*\in\CC^{\ell}$,
$${\rm HF}_{I_{p^*}}(d) := \binom{n+d}{d} - \dim_{\CC}((I_{p^*})_{\leq d}) = {\rm HF}_I(d).$$

\subsection{Numerical Hilbert functions}

With the setup described above, we develop a certified approach 
for computing lower bounds on the Hilbert
function of $I$ that combines
\cite{NumHilbert} with a method 
that guarantees the existence
of a solution to a polynomial system,
one such approach being $\alpha$-theory~\cite[Ch.~8]{alphaTheory}.  
We first introduce the exact computation
followed by the certified numerical approach
for computing lower bounds
using a well-constrained system. 

Suppose that $p^*\in\CC^\ell$ and $V_{p^*}=\{q_1^*,\dots,q_k^*\}\subset\Var({h_{1}}_{p^*},\dots,{h_{n}}_{p^*})\subset\CC^n$ consists of nonsingular solutions.  Let $I(V_{p^*})\subset\CC[x_1,\dots,x_n]$ be the ideal of polynomials vanishing on~$V_{p^*}$.  
By the implicit function theorem, each point $q_i^*\in V_{p^*}$
lifts locally to a solution in $\Var(I)$, say $q_i(p)$
with $q_i^* = q_i(p^*)$.

Let $\nu_d:\CC^n\rightarrow\CC^{\binom{n+d}{d}}$
be the degree $\leq d$ Veronese embedding, i.e.,
$$\nu_d(x) = \left[\begin{array}{ccccccccccccc}
1 & x_1 & \cdots & x_n & x_1^2 & x_1 x_2 & \cdots & x_n^2 & \cdots & x_1^d & x_1^{d-1} x_2 & \cdots & x_n^d\end{array}\right].$$
Moreover, let $M_d:\left(\CC^n\right)^k\rightarrow
\CC^{k\times \binom{n+d}{d}}$ be 
$$M_d(y_1,\dots,y_k) = \left[\begin{array}{c}
\nu_d(y_1) \\ \vdots \\ \nu_d(y_k) \end{array}\right].$$
With this setup, $\rank M_d(q_1^*,\dots,q_k^*)
\leq \rank M_d(q_1(p),\dots,q_k(p)) \leq {\rm HF}_I(d)$.
The following theorem adds certification 
to produce guaranteed lower bounds on $HF_I(d)$.

\begin{theorem}\label{thm:Numeric_Hilbert_LowerBound}
With the setup above, let $p^*\in\CC^\ell$, 
$q_1,\dots,q_k\in\CC^n$, and $S_d(q_1,\dots,q_k)$ 
be a $k \times k$ submatrix of $M_d(q_1,\dots,q_k)$.
Then, $HF_I(d) \geq k$ provided that
there exists a nonsingular solution to the well-constrained system
$$\sG(y_1,\dots,y_k,\Lambda) = \left[
\begin{array}{c}
G_{p^*}(y_1) \\ \vdots \\ G_{p^*}(y_k) \\ \Lambda\cdot S_d(y_1,\dots,y_k) - {\bf I}
\end{array}\right]$$
where $G_{p^*}(x) = [h_{1p^*}(x),\dots,h_{np^*}(x)]^T$ and $\bf I$ is the $k\times k$ identity~matrix.
\end{theorem}
\begin{proof}
Since $S_d(q_1,\dots,q_k)$ is a submatrix
of $M_d(q_1,\dots,q_k)$ and both have $k$ rows,
$$k \geq \rank M_d(q_1,\dots,q_k) \geq \rank S_d(q_1,\dots,q_k).$$
Suppose that $(q_1^*,\dots,q_k^*,\Lambda^*)$ is a nonsingular solution
of $\sG = 0$.  
Note that $\sG$ is well-constrained consisting
of $kn+k^2$ polynomials in $kn+k^2$ variables.
Thus, the structure of $\sG$ yields that 
each $q_i^*$ is a nonsingular solution of
$G_{p^*} = 0$ and $\Lambda^* = S_d(q_1^*,\dots,q_k^*)^{-1}$.
Therefore, 
$${\rm HF}_I(d) \geq \rank M_d(q_1^*,\dots,q_k^*) \geq \rank S_d(q_1^*,\dots,q_k^*) = k.$$
\end{proof}

In practice, one uses numerical approximations of
known solutions to select $k$ and the 
corresponding $k\times k$ submatrix which one expects to be invertible.
Therefore, starting from numerical approximations,
one can use a local certification routine,
e.g., $\alpha$-theory~\cite[Ch.~8]{alphaTheory},
to prove the existence of a nonsingular solution 
of $\sG = 0$ thereby certifying ${\rm HF}_I(d)\geq k$.

For each $d\geq 0$, let ${\rm HF}^{\rm num}_I(d)$ 
be the largest value $k$ for which one
can certify ${\rm HF}_I(d) \geq k$
based on known numerical data.
Hence, ${\rm HF}^{\rm num}_I(d) \leq {\rm HF}_I(d)$.

An upper bound on ${\rm HF}_I(d)$
is developed in Section~\ref{subSec:SymbolicReduction}
with an example in Section~\ref{subSection:HilbertExample}.

\subsection{Symbolic reduction}\label{subSec:SymbolicReduction}

One option for computing the Hilbert
function and thus $\deg(I)$ is to compute a Gr\"obner
basis for $I=\langle h_1,\dots, h_n\rangle\subset R$.  However, when this is not practical,
partial information provides upper bounds on the Hilbert
function.  For example, for each $d\geq0$ and $e\geq 0$,
consider the linear space 
$$J_d^e = \left(\Span_\kk\left(\bigcup_{i=1}^n \langle h_i \rangle_{\leq d+e}\right)\right)_{\leq d}.$$
Clearly, $J_d^e \subset I_{\leq d}$ showing that 
$${\rm HF}_I(d) \leq \binom{n+d}{d} - \dim_\kk J_d^e.$$
Moreover, $J_d^e\subset J_d^{e+1} \subset I_{\leq d}$ and 
there exists $e^*\geq 0$ such that $J_d^{e} = I_{\leq d}$ for all $e\geq e^*$
showing that this upper bound eventually becomes sharp.
In fact, one such stopping criterion is if~$e^*$ such that
${\rm HF}^{\rm num}_I(d) = \binom{n+d}{d}-\dim_\kk J_d^{e^*}$,
then ${\rm HF}^{\rm num}_I(d) = {\rm HF}_I(d) = \binom{n+d}{d}-\dim_\kk J_d^e$
for~$e\geq e^*$.

Let $p^*\in\QQ^{\ell}$. Rather than taking an exhaustive approach that considers
all possible polynomials, one could instead repeat
similar computations on $I$ to the modular computations 
performed when computing a Gr\"obner basis for 
$I_{p^*}$ over a finite field. 
Thus, one uses these computations to 
guide the computation of the polynomials in $I_{\leq d}$
via the Gr\"obner trace algorithm~\cite{GBTrace}.
Since one is performing exact computations on $I$, 
one maintains the upper bounds~on~${\rm HF}_I(d)$.

\subsection{Example}\label{subSection:HilbertExample}

We demonstrate the lower and upper bounds
via the parameterized system 
$$G(a,b;P) = P\cdot F(a,b)$$
where $P$ is a $6\times 14$ matrix of parameters
and $F$ is as in \eqref{eq:CubicsAffine}.
Let $I = \langle G\rangle$ and consider
$$P^* = \left[\begin{array}{rrrrrrrrrrrrrr}
1 & -2 & 2 & -4 & -4 & -5 & -3 & 1 & -1 & -1 & -2 & -3 & 1 & -5 \\
0 & 0 & 3 & 4 & 5 & -1 & -3 & -4 & -5 & -5 & 4 & -1 & -5 & -4 \\
-5 & -4 & -1 & 0 & -5 & -3 & -4 & 4 & -3 & 4 & -1 & -4 & -3 & 2 \\
-2 & 1 & -5 & 5 & 3 & 3 & -4 & 1 & -4 & 5 & -4 & -4 & -2 & 3 \\
-4 & -3 & -3 & -5 & 3 & -1 & 4 & -2 & -3 & 0 & 3 & 5 & 4 & 2 \\
3 & 2 & 5 & -1 & 4 & 5 & 1 & 0 & -3 & 0 & -1 & 5 & -5 & -1 
\end{array}\right]$$
with $G_{P^*}(a,b) = P^*\cdot F(a,b)$.
Using {\tt Bertini} \cite{Bertini,BertiniBook}, we computed
numerical approximations of~18 points in $\Var(G_{P^*})$.
From this data, we applied Theorem~\ref{thm:Numeric_Hilbert_LowerBound}
for $d = 1,2$ to yield
\begin{equation}\label{eq:HilbertNumF}
{\rm HF}_I^{\rm num}(1) = 7 \,\,\,\,\,\, \hbox{and} \,\,\,\,\,\,
{\rm HF}_I^{\rm num}(d) = 18 \hbox{~for~}d \geq 2
\end{equation}
as follows.  For $d = 1$, we selected $7$ points and 
utilized {\tt alphaCertified} \cite{alphaCertified}
with Theorem~\ref{thm:Numeric_Hilbert_LowerBound}
to certify 
$$7 = {\rm HF}_I^{\rm num}(1) \leq {\rm HF}_I(1) \leq \binom{6+1}{1} = 7.$$
Therefore, we know ${\rm HF}_I^{\rm num}(1) = {\rm HF}_I(1) = 7$.  

For $d = 2$, we utilized all $18$ known points 
with the subset of columns corresponding to the following $18$ monomials:
$$a_1, a_2, a_3, a_4, b_1, b_2, a_1b_1, a_1b_2, a_2a_4, a_2b_2, a_3^2, a_3a_4, a_3b_1, a_3b_2, a_4b_1, b_1^2, b_1b_2, b_2^2.$$
This certification via {\tt alphaCertified}
and Theorem~\ref{thm:Numeric_Hilbert_LowerBound}
shows that ${\rm HF}_I^{\rm num}(2) = 18$.
Since we utilized all $18$ points in this computation,
we have ${\rm HF}_I^{\rm num}(d) = 18$ for $d\geq 2$.

Next, we turn to upper bounds for ${\rm HF}_I(d)$
for $d = 2,3$.  The following tables summarize
the exhaustive approach for symbolic reduction described
in Section~\ref{subSec:SymbolicReduction} for $d=2,3$
and various values of $e\geq0$.
$$\begin{array}{cccc}
\begin{array}{c|c|c} 
e & \dim_{\kk} J_2^e & \binom{6+2}{2}-\dim_{\kk} J_2^e \\[0.03in] \hline
0 & 0 & 28 \\
1 & 3 & 25\\
2 & 3 & 25\\
3 & 5 & 23\\
4 & 9 & 19\\
5 & 10 & 18\\
\end{array} &&& 
\begin{array}{c|c|c} 
e & \dim_{\kk} J_3^e & \binom{6+3}{3}-\dim_{\kk} J_3^e \\[0.03in] \hline
0 & 6 & 78\\
1 & 25 & 59\\
2 & 38 & 46 \\
3 & 63 & 21 \\
4 & 66 & 18  
\end{array} 
\end{array}
$$
Hence, ${\rm HF}_I(2) \leq {\rm HF}_I(3) \leq 18$.
Combining with \eqref{eq:HilbertNumF} shows that
${\rm HF}_I(2) = {\rm HF}_I(3) = 18$
and thus ${\rm HF}_I(d) = 18$ for $d \geq 2$, i.e., 
${\rm HF}_I = {\rm HF}_I^{\rm num}$.
In particular, this shows that there are 
generically $18$ solutions to $G = 0$,
i.e., $g_0(X,\CC^6) = 18$ where $X = \Var(F)$
with $F$ as in \eqref{eq:CubicsAffine}.

\section{Probabilistic analysis}\label{sec:ProbRandomIsGeneral}

The theory in the previous sections relied upon general choices
of parameters.  In this section, we give explicit characterization 
of when constants are sufficiently general for our computations. 
Since Alt's problem is formulated affinely in Section~\ref{subSection:AltsIntro},
we focus on the affine case but note that 
the results proved here have analogs in the other 
settings discussed in Section~\ref{section:background}. 

Let $R=\QQ[x_1,\dots,x_n]$, $f=[f_1,\dots, f_r]^T\subset R$,
and $i\in \{0,\dots,n-1\}$.  Consider the ideal
\begin{equation}
I_i(\Theta,\Lambda,\mu)=\langle \Theta \cdot x -{\bf 1}, \Lambda\cdot f, 1 -(\mu\cdot f)T\rangle
\label{eq:I_lambda_theta}
\end{equation}
where $\Theta\in \QQ^{i\times n}$, $\Lambda\in \QQ^{(n-i)\times r}$, and $\mu\in \QQ^r$.  
The following is an affine version of \cite[Thm.~4.1]{H16}. 
\begin{theorem}
Let $f=[f_1,\dots, f_r]^T\subset R = \QQ[x_1,\dots,x_n]$, $i\in\{0,\dots,n-1\}$,
and $X = \Var(f)$.
For general choices of $\Theta\in \QQ^{i\times n}$, $\Lambda\in \QQ^{(n-i)\times r}$, and $\mu\in \QQ^r$,
$I_i(\Theta,\Lambda,\mu)\subset R[T]$ as in \eqref{eq:I_lambda_theta} is either the unit ideal or
has dimension $0$. In either case,
\begin{equation}
g_i(X,\CC^n)=\dim_{\QQ} \left(R[T]/I_i(\Theta,\Lambda,\mu)\right) \label{eq:g_iV2}
\end{equation} 
is independent of $\Theta$, $\Lambda$ and $\mu$ (provided they are general).  \label{thm:Computeg_i}
\end{theorem} \begin{proof}
This immediately follows from Bertini's Theorem and the Rabinowitz trick. 
The proof is identical to that of \cite[Thm.~4.1]{H16} 
with $\pp^n$ replaced by $\CC^n$. 
\end{proof}

In practice, when computing the numbers $g_i(X,\CC^n)$, 
the constants $\Theta$, $\Lambda$, and $\mu$ are chosen randomly from a finite set 
of rational numbers, e.g., all rational numbers of bounded height 
or a subset of floating-point numbers of a given bit size.
Therefore, even working over $\QQ$,
the computation of the integers $g_i(X,\CC^n)$ 
in Theorem~\ref{thm:Computeg_i} is probabilistic. 
For our random choices in the following,
we will utilize a uniform distribution on the chosen finite set. 

Since the results below depend on generators with
integer coefficients, we will assume
that, without loss of generality, 
$f=[f_1,\dots,f_r]^T\subset\ZZ[x_1,\dots,x_n]$
and that the coefficients in each $f_i$ are relatively prime.

\subsection{Lucky primes for randomly constructed ideals}

To simplify computations, we can employ modular methods
which require the notion of lucky primes as in \cite[\S5.1]{ModGB} and~\cite{Pauer}. 

The following is from \cite[Defn.~4.1]{Pauer}. 

\begin{definition}\label{def:PauerLucky}
For $h_1,\dots,h_k\in\ZZ[x_1,\dots,x_n]$,
let $I=\langle h_1,\dots, h_k \rangle\subset \QQ[x_1,\dots, x_n]$ and $I_\ZZ=\langle h_1,\dots, h_k \rangle\subset \ZZ[x_1,\dots, x_n]$ be ideals 
with Gr\"obner bases $G$ and $G_\ZZ$, respectively. 
A prime $p$ is {\em Pauer lucky} with respect to $I$ if and only if $p$ 
does not divide any of the leading coefficients of the polynomials in $G_\ZZ$. Otherwise, $p$ is {\em Pauer unlucky} with respect to $I$.
\end{definition}

For $h_1,\dots,h_k\in\ZZ[x_1,\dots,x_n]$ 
and $I= \langle h_1,\dots, h_k\rangle\subset \QQ[x_1,\dots, x_n]$,
let $LM_\QQ(I)$ denote the set of leading monomials of $I$. 
For a prime $p$, let $LM_{\ZZ_p}(I)$ denote the leading monomials of $\langle h_1,\dots, h_k \rangle\subset \ZZ_p[x_1,\dots, x_n]$. 
The following is \cite[Prop.~4.1]{Pauer} which states 
that if a prime $p$ is Pauer lucky, 
then the ideal of leading terms computed over $\QQ$ and $\ZZ_p$ 
agree.
\begin{proposition}[Prop.~4.1 of \cite{Pauer}]
With the setup as above, if a prime $p$ is
Pauer lucky with respect to $I$, then $LM_\QQ(I)=LM_{\ZZ_p}(I)$.
\end{proposition}

The following is used to avoid dividing any coefficients.

\begin{definition}\label{def:RealitivelyLarge}
For a finite collection of polynomials $h\subset\ZZ[x_1,\dots,x_n]$, we say that a prime~$p$ is {\em large relative to the coefficients of $h$} if $p$ is larger than
any coefficient appearing in a polynomial contained in $h$.
\end{definition}

The following shows that every prime
that is large relative to the coefficients of $f$
is Pauer lucky for $I_i(\Theta,\Lambda,\mu)$ defined 
in \eqref{eq:I_lambda_theta} on a Zariski dense set.

\begin{theorem}
Following the notation from Theorem~\ref{thm:Computeg_i}
such that $f\subset\ZZ[x_1,\dots,x_n]$,
suppose that $U$ is a Zariski dense subset of 
$\QQ^{i\times n}\times \QQ^{(n-i) \times r}\times \QQ^r$ where the statements of Theorem \ref{thm:Computeg_i} hold.
Also, suppose that for $(\Theta,\Lambda,\mu)\in U$, 
we have that $\dim(I_i(\Theta,\Lambda,\mu))=0$. 
Then, every prime $p$ which is large relative to the coefficients of $f$ is a Pauer lucky prime for 
the ideal $I_i(\Theta,\Lambda,\mu)$
 on a Zariski dense set of parameters 
$(\Theta,\Lambda,\mu)\in \QQ^{i\times n}\times \QQ^{(n-i) \times r}\times \QQ^r$.\label{Theorem:GeneralCOefficentsAllPrimesLucky}
\end{theorem}\begin{proof}
Since each entry of $\Theta, \Lambda$, and $\mu$ is rational, we symbolically
treat the numerator and denominator separately.  
To that end, let $\Theta^{\rm num},\Theta^{\rm denom}\in \ZZ^{i\times n}, \Lambda^{\rm num},\Lambda^{\rm denom}\in \ZZ^{(n-i) \times r}$ and $\mu^{\rm num},\mu^{\rm denom}\in \ZZ^r$ denote the corresponding numerators and denominators
of $\Theta, \Lambda$ and $\mu$, respectively. 
Let $\mathfrak{W}=(\Theta^{\rm num},\Theta^{\rm denom},\Lambda^{\rm num},\Lambda^{\rm denom},\mu^{\rm num},\mu^{\rm denom})$ denote the collection
of integer parameters.

For $I_i(\Theta,\Lambda,\mu)$ as in \eqref{eq:I_lambda_theta},
let $I_i(\mathfrak{W})$ be the ideal obtained by clearing denominators
in each of the generators of $I_i(\Theta,\Lambda,\mu)$.
Note that for any $(\Theta,\Lambda,\mu)\in U$, 
the corresponding $I_i(\Theta,\Lambda,\mu)$ and $I_i(\mathfrak{W})$
define the same ideal in $\QQ[x_1,\dots, x_n,T]$. Also note that since $p$ is large relative to the coefficients of $f$, 
each coefficient in $f$ is not divisible by $p$.

Given an integer vector $\mathfrak{W}$ where each entry
of $\Theta^{\rm denom}$, $\Lambda^{\rm denom}$, $\mu^{\rm denom}$
is nonzero, we let ${\rm rat}(\mathfrak{W})$ 
denote the corresponding rational values 
$(\Theta,\Lambda,\mu)\in \QQ^{i\times n}\times \QQ^{(n-i) \times r}\times \QQ^r$.

With this setup, we will consider $I_i(\mathfrak{W})$ as an ideal in the 
ring $\left(\ZZ[\mathfrak{W}]\right)[x_1,\dots,x_n,T]$.
Let ${G}\subset \left(\ZZ[\mathfrak{W}]\right)[x_1,\dots,x_n,T]$ 
be a Gr\"obner basis of $I_i(\mathfrak{W})$ consisting of, say, $\nu$ polynomials.
Let $\{c_1(\mathfrak{W}), \dots, c_\nu(\mathfrak{W}) \}$ be the leading coefficients of ${G}$ where each $c_j(\mathfrak{W})\in \ZZ[\mathfrak{W}]$.
Given $(\Theta,\Lambda,\mu)\in U$ and integer vector $\mathfrak{W}$
such that ${\rm rat}(\mathfrak{W}) = (\Theta,\Lambda,\mu)$,
the leading monomials of~$G$ evaluated at $\mathfrak{W}$
are a superset of the leading monomials obtained
by computing a Gr\"obner basis of $I_i(\mathfrak{W})$ in the ring $\ZZ[x_1,\dots,x_n,T]$.  In particular, the evaluated values of the leading coefficients
$\{c_1(\mathfrak{W}), \dots, c_\nu(\mathfrak{W}) \}$ are multiples of the 
leading coefficients of a Gr\"obner basis of 
$I_i(\mathfrak{W})$ in the ring $\ZZ[x_1,\dots,x_n,T]$. 
Therefore, for every $(\Theta,\Lambda,\mu)\in U$
giving rise to a corresponding $\mathfrak{W}$, 
either $p$ is Pauer lucky or Pauer unlucky. 
Hence, there must be a Zariski dense subset of $U$ such that at least one of these properties is satisfied. 
It is enough to show that there is a Zariski dense subset of $U$ such that for each $(\Theta,\Lambda,\mu)$ in this subset
with corresponding $\mathfrak{W}$, we have that $p$ does not divide 
$\{c_1(\mathfrak{W}), \dots, c_\nu(\mathfrak{W}) \}$ since this 
shows that $p$ does not divide the leading coefficients of a Gr\"obner basis of  $I_i(\mathfrak{W})$ in the ring $\ZZ[x_1,\dots,x_n,T]$.

We note that if $(\Theta,\Lambda,\mu)\in U$
with corresponding $\mathfrak{W}$, 
$\dim_\QQ(I_i(\Theta,\Lambda,\mu))=\dim_\QQ(I_i(\mathfrak{W}))=0$.
In particular, each leading monomial of $G$ is not constant
with respect to $x_1,\dots,x_n,T$.

For the sake of reaching a contradiction, suppose that there exists a Zariski dense subset $D\subset U$ such that for all $(\Theta,\Lambda,\mu)\in D$
with corresponding integer vector $\mathfrak{W}$, 
we have that~$p$ is a Pauer {\em unlucky} prime for the ideal
$I_i(\mathfrak{W}) \subset \ZZ[x_1,\dots,x_n,T]$. 
If $p$ is Pauer {\em unlucky} for $I_i(\mathfrak{W})$ for all ${\rm rat}(\mathfrak{W})\in D$, then there must exist a fixed $j\in\{1,\dots,\nu\}$ and a (possibly smaller) Zariski dense subset $\tilde{D}\subset D$ such that $p$ divides $c_j(\mathfrak{W})$ for all $\mathfrak{W}$ corresponding to $(\Lambda,\Theta,\mu)\in \tilde{D}$.

Since $\tilde{D}\subset \QQ^{i\times n}\times \QQ^{(n-i) \times r}\times \QQ^r$ is Zariski dense, it is also dense in the classical topology. 
This means that $c_j(\mathfrak{W})$ must be constant since a nonconstant polynomial must map dense sets to dense sets and the set of 
multiples of a prime is not dense in $\QQ$ (in the classical topology). 
It follows that $c_j(\mathfrak{W})=\tilde{c}\in \ZZ$ for all ${\rm rat}(\mathfrak{W})\in \QQ^{i\times n}\times \QQ^{(n-i) \times r}\times \QQ^r$.
This is not possible due to the construction of $I_i(\mathfrak{W})$ since
every non-constant term of every polynomial in $I_i(\mathfrak{W})$ is 
multiplied by some parameter value, i.e., at least one parameter must appear in 
each leading term since the ideal $I_i(\mathfrak{W})$ is not generated by a constant.
Hence, there is {\em no} dense subset of~$U$ such that $p$ is Pauer {\em unlucky}.  
It follows that there must exist some dense subset of $W\subset U\subset \QQ^{i\times n}\times \QQ^{(n-i) \times r}\times \QQ^r $ for which $p$ is a Pauer {\em lucky} prime for the ideal $I_i(\Theta,\Lambda,\mu)$ for all $(\Theta,\Lambda,\mu)\in W$. 
\end{proof}
We now illustrate the construction in the proof of Theorem \ref{Theorem:GeneralCOefficentsAllPrimesLucky} with an example. 
\begin{example}
Reconsider $f(x) = [x_1,x_2,x_1x_2^2,x_1^3x_2^2]^T$ from 
Example~\ref{example:firstEx}
where $g_1(X,\CC^2)=5$ with $\mbox{$X=\Var(f) = \{(0,0)\}$}$. 
Note that every prime $p\geq 2$ is 
large relative to the coefficients of~$f$ 
as in Definition \ref{def:RealitivelyLarge}.
Since the ideal
\begin{equation}\label{eq:I_ex}
\left\langle \theta_1 x_1+ \theta_2 x_2-1,  \lambda_1 x_1+\lambda_2 x_2+\lambda_3 x_1x_2^2+\lambda_4 x_1^3x_2^2\right\rangle
\end{equation}
generically already has degree $5=g_1(X,\CC^2)$, 
we can simplify the presentation
by just considering $I_1(\Theta,\Lambda)$ equal
to the ideal in \eqref{eq:I_ex}.
Define 
$$\mathfrak{W}=[
\theta_1^{\rm num},\theta_1^{\rm denom},
\theta_2^{\rm num},\theta_2^{\rm denom},
\lambda_1^{\rm num},\lambda_1^{\rm denom},
\lambda_2^{\rm num},\lambda_2^{\rm denom},
\lambda_3^{\rm num},\lambda_3^{\rm denom},
\lambda_4^{\rm num},\lambda_4^{\rm denom}]$$
so that $I_i(\Theta,\Lambda)$ is equivalent to
$$
\left\langle \frac{\theta_1^{\rm num}}{\theta_1^{\rm denom}}x_1+\frac{\theta_2^{\rm num}}{\theta_2^{\rm denom}}x_2-1,  \frac{\lambda_1^{\rm num}}{\lambda^{\rm denom}} x_1+\frac{\lambda_2^{\rm num}}{\lambda^{\rm denom}_2} x_2+\frac{\lambda_3^{\rm num}}{\lambda^{\rm denom}_3}x_1x_2^2+\frac{\lambda_4^{\rm num}}{\lambda^{\rm denom}_4}x_1^3x_2^2\right\rangle.
$$
Symoblically clearing denominators, we obtain the ideal $I_i(\mathfrak{W}) \subset (\ZZ[\mathfrak{W}])[x_1,x_2]$, namely
$$
\left\langle 
\begin{array}{l}
{\theta_1^{\rm num}}{\theta_2^{\rm denom}}x_1+{\theta_2^{\rm num}}{\theta_1^{\rm denom}}x_2-{\theta_1^{\rm denom}}{\theta_2^{\rm denom}}, \\
{\lambda_1^{\rm num}}{\lambda^{\rm denom}_2}{\lambda^{\rm denom}_3}{\lambda^{\rm denom}_4} x_1+{\lambda_2^{\rm num}}{\lambda^{\rm denom}_1}{\lambda^{\rm denom}_3}{\lambda^{\rm denom}_4} x_2\\
\,\,\,\,\,\,\,\,\,\,\,\,\,\,\,\,\,\,\,\,\,\,\,\,\,\,\,\,\,\,+~{\lambda_3^{\rm num}}{\lambda^{\rm denom}_1}{\lambda^{\rm denom}_2}{\lambda^{\rm denom}_4}x_1x_2^2 
+{\lambda_4^{\rm num}}{\lambda^{\rm denom}_1}{\lambda^{\rm denom}_2}{\lambda^{\rm denom}_3}x_1^3x_2^2
\end{array}\right\rangle.
$$
Using lexicographic order, the leading terms of 
a Gr\"obner basis for $I_i(\mathfrak{W})$ are
\begin{equation}\label{eq:LeadTermsEx}
\left\{c_1(\mathfrak{W})x_2^5,~c_2(\mathfrak{W}) x_1,~
c_3(\mathfrak{W}) x_1 x_2^4,~
c_4(\mathfrak{W}) x_1^2 x_2^3,~
c_5(\mathfrak{W}) x_1^3 x_2^2
\right\}
\end{equation}
where
$$
\begin{array}{rcl}
c_1(\mathfrak{W}) &=& 
\left({\theta^{\rm denom}_{1}}
{\theta^{{\rm num}}_{2}}\right)^{3}
{\lambda^{\rm denom}_{1}}{\lambda^{\rm denom}_{2}}{\lambda^{\rm denom}_{3}}{\lambda}^{\rm num}_{4},\\
c_2(\mathfrak{W}) &=& {\theta^{{\rm denom}}_{2}}{\theta^{\rm num}_{1}},\\
c_3(\mathfrak{W}) &=& 
\left({\theta^{{\rm denom}}_{1}}{\theta^{\rm num}_{2}}\right)^{2}{\lambda^{\rm denom}_{1}}{\lambda^{{\rm denom}}_{2}}{\lambda^{{\rm denom}}_{3}}{\lambda}_{4}^{\rm num}, \\
c_4(\mathfrak{W}) &=& {\theta^{{\rm denom}}_{1}}{\theta}^{\rm num}_{2}{\lambda^{{\rm denom}}_{1}}{\lambda^{{\rm denom}}_{2}}{\lambda^{{\rm denom}}_{3}}{\lambda}_{4}^{\rm num},\\
c_5(\mathfrak{W}) &=& {\lambda^{{\rm denom}}_{1}}{\lambda^{\rm denom}_{2}}{\lambda^{{\rm denom}}_{3}}{\lambda}_{4}^{\rm num}.
\end{array}
$$
Generically, \eqref{eq:LeadTermsEx} yields that the
$5=g_i(X,\CC^2)$ standard monomials are $1,x_2,\dots,x_2^4$.
As expected, each leading coefficient $c_i(\mathfrak{W})$ 
is nonconstant.

For illustration, consider $\Theta=[3/2,2/3]$ and $\Lambda=[7/5,9/11,-5/13,13/17]$
corresponding with $\mathfrak{W} = [3,2,2,3,7,5,9,11,-5,13,13,17]$.
Evaluating \eqref{eq:LeadTermsEx} yields the following leading terms:
\begin{equation}
\left\{9295000 {x}_{2}^{5},~33 {x}_{1},~929500 {x}_{1}{x}_{2}^{4},~92950 {x}_{1}^{2}{x}_{2}^{3},~9295 {x}_{1}^{3}{x}_{2}^{2} \right\}.
\label{eq:LCv1}
\end{equation}
On the other hand, if we substitute $\Theta=[3/2,2/3]$ and $\Lambda=[7/5,9/11,-5/13,13/17]$ into~\eqref{eq:I_ex},
clear denominators, and compute a Gr\"obner basis for the resulting ideal in $\ZZ[x_1,x_2]$, the leading terms are \begin{equation}\label{eq:LCv2}
\left\{845000 {x}_{2}^{5},~33 {x}_{1},~{x}_{1}{x}_{2}^{4},~{x}_{1}^{2}{x}_{2}^{3},~11 {x}_{1}^{3}{x}_{2}^{2}\right\}.
\end{equation}
As expected, each leading coefficient in \eqref{eq:LCv2} 
divides the corresponding coefficient in \eqref{eq:LCv1}, e.g., $9295000=845000\cdot 11$ and $9295=11\cdot 845$.  
For this problem, both \eqref{eq:LCv1} and \eqref{eq:LCv2}
yield that the same set of Pauer unlucky primes
via Definition~\ref{def:PauerLucky}, namely $\{2, 3, 5, 11, 13\}$.
\end{example}

\subsection{Probability of a prime being unlucky for randomly constructed ideals} 

Instead of considering a Zariski dense subset
of $\QQ^{i\times n}\times \QQ^{(n-i) \times r}\times \QQ^r$
as in Theorem~\ref{Theorem:GeneralCOefficentsAllPrimesLucky},
we now consider random choices from a finite subset.

\begin{theorem}
Let $f=[f_1,\dots, f_r]^T\subset \ZZ[x_1,\dots,x_n]$,
$i\in \{0,\dots,n-1\}$, and $p$ be a prime
that is large relative to the coefficients of $f$.
Suppose that $S=\{0,\dots, p-1\}$
and each entry of $(\Theta,\Lambda,\mu)\in \QQ^{i\times n}\times \QQ^{(n-i)\times r}\times \QQ^r$ is randomly selected from $S$.
Then, a lower bound on 
the probability that $p$ is a Pauer lucky prime for 
$I_i(\Theta,\Lambda,\mu)$ as in \eqref{eq:I_lambda_theta} is 
\begin{equation}\label{eq:Probability}
P(p{\rm \; is \; Pauer \; lucky \; for \;}I_i(\Theta,\Lambda,\mu))\geq \left(\frac{p-1}{p}\right)^{\nu},
\end{equation}
where $\nu$ is the number of polynomials in a Gr\"obner basis of $I_i(\Theta,\Lambda,\mu)$ in $\ZZ[x_1,\dots,x_n,T]$ and $$
\nu \leq {\deg(I_i(\Theta, \Lambda, \mu)) +(n+1) \choose (n+1)}.
$$
\label{thm:ProbPauerLucky}
\end{theorem}\begin{proof}
Consider $I_i(\Theta,\Lambda,\mu)$ as an ideal in the ring $(\ZZ[\Theta,\Lambda,\mu])[x_1,\dots,x_n,T]$ and let ${G}$ be a Gr\"obner basis of $I_i(\Theta,\Lambda,\mu)$ in this ring which consists of $\nu$ polynomials. 

First, suppose the ideal $I_i(\Theta,\Lambda,\mu)$ in $(\ZZ[\Theta,\Lambda,\mu])[x_1,\dots,x_n,T]$ is not generated by a constant. 
From the proof of Theorem~\ref{Theorem:GeneralCOefficentsAllPrimesLucky}, 
we have that the leading coefficients of ${G}$, say
$\{c_1(\Theta,\Lambda,\mu), \dots, c_\nu(\Theta,\Lambda,\mu) \}$, are nonconstant polynomials in $\ZZ[\Theta,\Lambda,\mu]$. Since a nonconstant polynomial must map $\ZZ_p$ to itself injectively, the probability that a given coefficient will be 
nonzero in $\ZZ_p$ where each entry
of $(\Theta,\Lambda,\mu)$ is chosen from $S$ is at least $(p-1)/p$. 

Since no leading monomial can have degree greater than $\deg(I_i(\Theta, \Lambda, \mu))$ and there are $$D={\deg(I_i(\Theta, \Lambda, \mu)) +(n+1) \choose (n+1)}$$ monomials having degree at most $\deg(I_i(\Theta, \Lambda, \mu))$ in $(n+1)$ variables, it follows that $\nu\leq D$.

Finally, if $I_i(\Theta,\Lambda,\mu)$ is generated by a constant, i.e., an element of $\ZZ[\Theta,\Lambda,\mu]$,
the same lower bound on the probability of success holds.
\end{proof}

\begin{remark}
Since the degree of $I_i(\Theta,\Lambda,\mu)$ is the quantity of interest,
one could produce lower bounds on the probability in \eqref{eq:Probability}
by using any upper bound on the degree of~$I_i(\Theta,\Lambda,\mu))$,~e.g., the multihomogeneous B\'ezout bound or the mixed volume.  Further, note that \cite[Cor.~2.1]{FGLM} yields the number of polynomials in a reduced Gr\"obner basis of a zero dimensional ideal over a field is at most 
the product of the number of variables and
the degree of the ideal. While the ideal in Theorem \ref{thm:ProbPauerLucky} is zero dimensional, 
it is not computed over a field.
However, it seems likely that there is a much lower upper bound on $\nu$ than the bound in Theorem \ref{thm:ProbPauerLucky}.
\end{remark}

In Theorem~\ref{thm:ProbPauerLucky}, each entry of $(\Theta,\Lambda,\mu)$ 
is taken from a finite set $\mathcal{S}=\{0,1,\dots,p-1\}$ 
while the ideal $I_i(\Theta,\Lambda,\mu)$ as in \eqref{eq:I_lambda_theta}
is defined over $\ZZ$.
We now turn to treating $I_i(\Theta,\Lambda,\mu)$ as an ideal
defined over a finite field.
To that end, let $g_i(X,\CC^n)$ denote the (correct) integer value 
as in \eqref{eq:g_iV2} computed using general rational values 
for $\Theta$, $\Lambda$, and $\mu$ as in Theorem~\ref{thm:Computeg_i}. 
Let $g_i(X,\CC^n)_{\mathcal{S}\subset \QQ}^{\rm rand}$ denote 
$\dim_{\QQ} \left(R[T]/I_i(\Theta,\Lambda,\mu)\right)$ 
computed where the entries of $\Theta$, $\Lambda$, and $\mu$ 
are randomly selected from a finite set $\mathcal{S}\subset \QQ$. 
Finally, let $g_i(X,\CC^n)_{\ZZ_p}^{\rm rand}$ denote 
$\dim_{\ZZ_p} \left(R[T]/I_i(\Theta,\Lambda,\mu)\right)$ 
where each entry of $\Theta$, $\Lambda$, and $\mu$ is
selected randomly among the elements of $\ZZ_p$.
The following provides a comparison of probabilities
for computing the same number.  In this statement, 
$P(A=B)$ denotes the probability that random variables 
$A$ and $B$ take on the same value.

\begin{theorem}
Following the setup of Theorem~\ref{thm:ProbPauerLucky},
$$
P(g_i(X,\CC^n)=g_i(X,\CC^n)_{\ZZ_p}^{\rm rand})\geq \left( \frac{p-1}{p}\right)^\nu P(g_i(X,\CC^n)=g_i(X,\CC^n)_{\mathcal{S}\subset \QQ}^{\rm rand})
$$ where $\nu$ is as in Theorem \ref{thm:ProbPauerLucky}. \label{thm:ProbZZpVsQQ}
\end{theorem}\begin{proof}
This follows immediately from Theorem \ref{thm:ProbPauerLucky}. 
\end{proof}

\subsection{Probability of choosing general coefficients}\label{SubSection:ProbConstantsGeneral}

Although Theorem~\ref{thm:Computeg_i} states 
that general values of the parameters yield the correct result,
it does not provide insight into the set of parameters
where \eqref{eq:g_iV2} fails.
The following provides degree bounds which are used
to develop lower bounds on the probability of choosing
general coefficients.
We note that the results in this section are essentially an 
affine version of those in \cite[\S3.3,~A.2]{H17}. 

As above, suppose that $f=[f_1,\dots, f_r]^T\subset \ZZ[x_1,\dots,x_n]$, $i\in \{0,\dots,n-1\}$, and $X = \Var(f)$.
Define the projection of $\CC^n$ along $X$ as the rational map 
$pr_X:\CC^n\to \CC^r$ given by $pr_X(z)=f(z)$. 
Let
$g_i(X,\CC^n)$ be the value in \eqref{eq:g_iV2}
for general $(\Theta, \Lambda, \mu)$ and 
let $g_i(X,\CC^n)(\Theta, \Lambda, \mu)$ be the value obtained for a particular choice of $(\Theta, \Lambda,\mu)$. 

Consider the algebraic sets \begin{equation}
\Gamma_0=\left\lbrace (x,y)  \;|\; f_1(x)\neq 0 , \; y=pr_X(x)\right\rbrace \subset \CC_x^n\times \CC_y^r \,\,\,\,\hbox{~~and~~}\,\,\,\, \Gamma=\overline{\Gamma_0}\subset \CC_x^n\times \CC_y^r. 
\end{equation} 
We call $\Gamma$ the graph of $pr_X$. 
Define $\Gamma_{\rm b}=\Gamma \cap (\Var(f_1) \times \CC^r)$. 
In this notation, $g_i(X,\CC^n)$ is simply 
the number of solutions (counted with multiplicity) to the system 
$$
(x,y)\in \Gamma, \;\; \Theta \cdot x -{\bf 1}=0,\;\; \Lambda\cdot y=0   
$$ 
where $\Theta \in \QQ^{i\times n}$ and $\Lambda \in \QQ^{(n-i)\times r}$ are general matrices. 

We now treat $\Theta$ and $\Lambda$ as variables and work in 
$\CC_x^n\times \CC_y^r\times \CC_\Theta^{i\times n}\times \CC_\Lambda^{(n-i)\times r} $.  We construct the corresponding discriminant variety as follows.
Let \begin{equation}
\Phi_{\rm b}=\left\lbrace (x,y,\Theta, \Lambda)\; | \; (x,y) \in \Gamma_{\rm b} , \; \Theta \cdot x - {\bf 1} = 0, \; \Lambda\cdot y=0\right\rbrace
\end{equation} 
and $\pi_2:\CC_x^n\times \CC_y^r\times \CC_\Lambda^{(n-i)\times r} \times \CC_\Theta^{i\times n}\to \CC_\Lambda^{(n-i)\times r}\times \CC_\Theta^{i\times n}$ be the projection onto the $\Theta$ and $\Lambda$ coordinates. 
Also, define 
\begin{equation}
\Phi_\Gamma=\left\lbrace (x,y,\Theta, \Lambda)\; | \; (x,y) \in \Gamma , \; \Theta \cdot x - {\bf 1}=0, \; \Lambda \cdot y=0\right\rbrace,
\end{equation}
and set $$
\mathfrak{D}=\det \begin{bmatrix} 
     {\rm Jac}_{x,y}(y_1-f_1(x),\dots,y_r-f_r(x)) \\
      \begin{array}{cc}
       \Lambda~~  & ~~0
       \end{array}  \\
     \begin{array}{lr}
       0 ~~& ~~\Theta
       \end{array} 
  \end{bmatrix}\in \CC[x,y,\Theta, \Lambda].
$$
The discriminant (as shown below in Proposition \ref{Prop:DescriminateVar}) is given by \begin{equation}\label{eq:Discriminant}
\Delta_{pr_X}=\begin{cases} \pi_2(\Phi_{\rm b})\cup \pi_2(\Var(\mathfrak{D})\cap \Phi_\Gamma) & {\rm if\;} \pi_2|_{\Phi_\Gamma}{\rm \; is\; surjective,}\\\pi_2(\Phi_\Gamma)& {\rm if\;} \pi_2|_{\Phi_\Gamma}{\rm \; is\;not\; surjective,}
\end{cases}
\end{equation} 
that is, we show that for $(\Theta, \Lambda)\not \in \Delta_{pr_X}$, 
we have that $g_i(X,\CC^n)=g_i(X,\CC^n)(\Theta, \Lambda)$. 
The following helps to show this.

\begin{lemma}
Let $(\Theta, \Lambda)\in (\CC_\Lambda^{(n-i)\times r} \times \CC_\Theta^{i\times n})\backslash \pi_2(\Phi_{\rm b})$. 
Suppose that 
$$B = \pi_2^{-1}(\Theta, \Lambda)\cap \Phi_\Gamma\subset \CC_x^n\times \CC_y^r\times \CC_\Lambda^{(n-i)\times r} \times \CC_\Theta^{i\times n}$$ and let 
$\widetilde{B}=\Var(\Theta \cdot x -{\bf 1},\Lambda \cdot f)\backslash X\subset \CC_x^n$. Then, all points in $B$ are such that $f_1(x)\neq 0$ and the map $(x,y)\mapsto x$ gives a bijection $B\to \widetilde{B}$ with inverse map $x\mapsto (x,pr_X(x))$. \label{lemma:fiber}
\end{lemma}
\begin{proof}
Note that $\Gamma$ may be defined as the Zariski closure of the image of the restriction of $pr_X$ to either $\CC^n\backslash \Var(f_1)$ or $\CC^n\backslash X$ since these are the same. 
Let $b\in \widetilde{B}$. 
Since $b\not\in X$, we can set $y=pr_X(b)$ and hence $(b,y)\in \Gamma$. 
It follows that $(b,y,\Theta, \Lambda)\in \Phi_\Gamma$ and is thus contained~in~$B$. 

By assumption, $(\Theta, \Lambda)\in (\CC_\Lambda^{(n-i)\times r} \times \CC_\Theta^{i\times n})\backslash \pi_2(\Phi_{\rm b})$ so take $(x,y,\Theta, \Lambda)\in B\subset \Phi_\Gamma$ to be a point lying over $(\Theta, \Lambda)$. Since $(\Theta, \Lambda)\not\in \pi_2(\Phi_{\rm b})$, this implies that $(x,y)\in \Gamma_0$ (since $\Gamma_0$ and $\Gamma_{\rm b}$ must be disjoint and $\Gamma=\Gamma_{\rm b}\cup \Gamma_0$). Hence, $y=pr_X(x)$ which implies that $x\in \widetilde{B}$. 
\end{proof}
\begin{proposition}\label{Prop:DescriminateVar}
Using the notation above, $g_i(X,\CC^n)=g_i(X,\CC^n)(\Theta, \Lambda)$ for~$\mbox{$(\Theta, \Lambda)\not \in \Delta_{pr_X}$}$. 
\end{proposition}\begin{proof}
The proof is split into two cases: either the restriction of $\pi_2$ to $\Phi_\Gamma$ is surjective or not. 

Suppose that $\pi_2$ restricted to $\Phi_\Gamma$ is not surjective. 
Note that since the map $\pi_2$ is closed, then we must have that $\Delta_{pr_X}$ has codimension at least one.  If $(\Theta, \Lambda)\not\in\Delta_{pr_X}$, then we have that there is no $(x,y)$ such that $(x,y,\Theta, \Lambda)\in \Phi_\Gamma$. 
Note that $\pi_2(\Phi_{\rm b})\subset \pi_2(\Phi_\Gamma)$ by definition.
By Lemma~\ref{lemma:fiber}, it follows that $\Var(\Theta \cdot x -{\bf 1},\Lambda \cdot f)\backslash X\subset \CC_x^n$ is empty. 
Hence, $g_i(X,\CC^n)=g_i(X,\CC^n)(\Theta, \Lambda)=0$.

Suppose that $\pi_2$ restricted to $\Phi_\Gamma$ is surjective. Let $(\Theta, \Lambda)\in (\CC_\Lambda^{(n-i)\times r} \times \CC_\Theta^{i\times n})\backslash \pi_2(\Phi_{\rm b})$. 
By \cite[Lemma A.11]{H17}, we have that $(\Theta, \Lambda)$ is a regular value of the restriction of $\pi_2$ to $\Gamma$ if and only if for any $(x,y,\Theta, \Lambda)$ in the fiber $\pi_2^{-1}(x,y,\Theta, \Lambda)\cap \Phi_\Gamma$, 
we have that $\mathfrak{D}(x,y,\Theta, \Lambda)\neq 0$. Hence, for $(\Theta, \Lambda)\in \left(\CC_\Lambda^{(n-i)\times r} \times \CC_\Theta^{i\times n} \right)\backslash \Delta_{pr_X}$ we have that $(\Theta, \Lambda)$ is a regular value of the restriction of $\pi_2$ to $\Gamma$. Then, since $\pi_2$ restricted to $\Phi_\Gamma$ is surjective, it follows that $\pi_2|_{\Phi_\Gamma}$ is a dominant map, and hence the fibers of its regular values have the same cardinality (e.g., see \cite[Prop. 12.19, Cor. 12.20, (12.6.2)]{AGI}). The conclusion follows by Lemma~\ref{lemma:fiber}. 
\end{proof}

The following provides a degree bound.

\begin{proposition}
Let $V=\Var(y_1-f_1(x),\dots,y_r-f_r(x))\subset \CC_x^n\times \CC_y^r$,
$D_{\min}=\min_{j} \deg(f_j)$, and $D_{\max}=\max_j \deg(f_j)$. 
Then, there exists a polynomial $Q\in \CC[\Theta, \Lambda]$ of degree at most 
$$2^n\cdot\left(D_{\min} + (r+n)\cdot D_{\max}\right)\cdot \deg(V)$$ 
such that if $Q(\Theta, \Lambda)\neq 0$, then $g_i(X,\CC^n)=g_i(X,\CC^n)(\Theta, \Lambda)$. \label{Prop:AffineDiscriminantBound}
\end{proposition}
\begin{proof}
We first consider the case where $\pi_2$ restricted to $\Phi_\Gamma$ is surjective. 
Note that we may choose the polynomial of minimal degree to be $f_1$ 
when defining $\Gamma_0$, i.e., assume \mbox{$D_{\min} = \deg(f_1)$}. 
We must bound the degree of the discriminant $\Delta_{pr_X}=\pi_2(\Phi_{\rm b})$. 
By B\'ezout's Theorem, 
$$
\deg(\Phi_b)\leq \deg(\Gamma_{\rm b})\cdot \deg(\Theta \cdot x -{\bf 1})\cdot \deg(\Lambda \cdot y)=2^n\cdot \deg(\Gamma_{\rm b}). 
$$ 
We next bound $\deg(\Gamma_{\rm b})$.
Let $\Omega = \left(\CC^n\setminus\Var(f_i)\right)\times \CC^r$.
For $G(x,y)=y-f(x)$,
we have $\Omega\cap \Gamma =\Omega\cap \Var(G)$. 
It follows that $\Gamma$ is one of the irreducible components 
of $V=\Var(G)$.  Hence, $\deg(\Gamma)\leq \deg(V)$.
Moreover, 
$$\Gamma_{\rm b}=\Gamma \cap (\Var(f_1) \times \CC^r),$$ 
showing that $\deg(\Gamma_{\rm b})\leq D_{\min} \cdot \deg(V)$. 
Therefore, $\deg(\Phi_{\rm b})\leq  2^n\cdot D_{\min} \cdot \deg(V)$.
Since degrees cannot increase under projection, it follows that 
$\deg(\Delta_{pr_X})\leq \deg(\pi_2(\Phi))\leq 2^n\cdot D_{\min} \cdot \deg(V)$. 

Next, we need to bound $\deg(\pi_2(\Var(\mathfrak{D})\cap \Phi_\Gamma)) \leq \deg(\Var(\mathfrak{D})\cap\Phi_\Gamma)$.
The polynomial $\mathfrak{D}$ has degree at most 
$(r+n)\cdot D_{\max}$ and $\deg(\Phi_\Gamma) \leq 2^n\cdot \deg(V)$.
The result now follows from \eqref{eq:Discriminant}.

The nonsurjective case follows similarly 
with $\deg(\Delta_{pr_X})\leq 2^n \cdot \deg(V)$.
\end{proof}

With this upper bound on degree, we can derive probabilistic bounds.

\begin{proposition}
With the setup from Proposition~\ref{Prop:AffineDiscriminantBound},
if the values for $\Lambda$, $\Theta$, and $\mu$ are chosen 
randomly from a finite set of $S\subset \QQ$, 
then $$
P(g_i(X,\CC^n)=\dim_\CC(R/I_i(\Theta, \Lambda,\mu))\geq 1-\frac{
g_i(X,\CC^n) + 2^n \cdot \left(D_{\min}+(r+n)\cdot  D_{\max}\right)\cdot \deg(V) }{|S|}.
$$\label{prop:AffineProbRandIsGen}
\end{proposition}\begin{proof}
Consider the ideal
\begin{equation}
J_i(\Theta, \Lambda)=\langle\Theta \cdot x -{\bf 1}, \Lambda \cdot f\rangle:\langle f \rangle^\infty.
\end{equation} 
Let $Q\in \CC[\Theta, \Lambda]$ be as in Proposition~\ref{Prop:AffineDiscriminantBound}
and suppose $\Theta$ and $\Lambda$ 
are selected such that $Q(\Theta, \Lambda)\neq 0$. 
Then, we know $\dim(J_i(\Theta, \Lambda))=0$ and $s:=g_i(X,\CC^n)=\deg(J_i(\Theta, \Lambda))$. 
Suppose that $\Var(J_i(\Theta, \Lambda))=\{q_1,\dots, q_s\}$ and 
consider the following polynomial of degree $s = g_i(X,\CC^n)$:
$$
U(\mu)=(\mu_1f_1(q_1)+\cdots+\mu_rf_r(q_1))\cdots (\mu_1f_1(q_s)+\cdots+\mu_rf_r(q_s))\in \CC[\mu_1,\dots, \mu_r].
$$
Note that if $U(\mu)\neq 0$, then $\deg(J_i(\Theta, \Lambda))=\deg(I_i(\Theta, \Lambda,\mu))$. 
The probability bound follows from applying the Schwartz-Zippel Lemma \cite{Schwartz,Zippel} using the polynomial $U\cdot Q \in \CC[\Theta, \Lambda,\mu]$. 
\end{proof}

The following considers the computation over finite fields.
 
\begin{corollary}
Let $f=[f_1,\dots, f_r]^T\subset \ZZ[x_1,\dots,x_n]$, 
$i\in \{0,\dots,n-1\}$,
and $p$ be a prime that is large relative to the coefficients of $f$.  Set $D_{\max} = \max_j \deg(f_j)$, 
$D_{\min} = \min_j \deg(f_j)$,~and
$$V=\Var(y_1-f_1(x),\dots,y_r-f_r(x))\subset \CC_x^n\times \CC_y^r.$$ 
Suppose that $$g_i(X,\CC^n)_{\ZZ_p}^{\rm rand}=\dim_{\ZZ_p}(R/I_i(\Theta, \Lambda,\mu))$$ is computed over $\ZZ_p$ 
where each entry of $\Theta$, $\Lambda$, and $\mu$ are selected
randomly from $\ZZ_p$.  Then,
$$
\hbox{\small $P(g_i(X,\CC^n)_{\ZZ_p}^{\rm rand}=g_i(X,\CC^n))\geq \left(\dfrac{p-1}{p} \right)^\nu \left(1-\dfrac{g_i(X,\CC^n) + 2^n \cdot \left(D_{\min}+(r+n)\cdot D_{\max}\right) \cdot \deg(V)}{p} \right)$}
$$where $\nu$ is as in Theorem~\ref{thm:ProbPauerLucky}
with the upper bound $$
 \nu \leq {\deg(I_i(\Theta, \Lambda, \mu)) +(n+1) \choose (n+1)}.
$$
\label{Corr:ProbBound_Affine_OverFiniteFeild}
 \end{corollary}\begin{proof}
 This follows immediately from Theorem \ref{thm:ProbZZpVsQQ} and Proposition \ref{prop:AffineProbRandIsGen}.
 \end{proof}

\section{Computations}\label{sec:Computations}

The following considers problems associated with Alt's problem
introduced in Section~\ref{subSection:AltsIntro}.
We let $f = [f_1,\dots,f_{15}]^T$ as in Section~\ref{subSection:AltsIntro}
which depend upon variables 
$\mathbf{x}=[a,\bar{a},b,\bar{b},x,\bar{x},y,\bar{y}]^T$
with $X=\Var(f)\subset \CC^8$ as in \eqref{eq:AltBaseLocus}. Throughout this section, 
we will suppose that all primes~$p$ being considered are such that $p\geq 5$
so that $p$ is large relative to the coefficients of $f$ 
as in Definition~\ref{def:RealitivelyLarge}. 
The following finds primes that
are large enough to provide a high probability 
that $g_i(X,\CC^8)_{\ZZ_p}^{\rm rand} = g_i(X,\CC^8)$.
The modular computations 
described below utilized {\tt Magma}~\cite{Magma}.

\subsection{Probabilistic considerations} \label{subSec:Prob}

The following specialize results from Section~\ref{sec:ProbRandomIsGeneral}
to the setup related to Alt's problem for computing
the numbers $g_i(X,\CC^8)$.  In particular, we provide theoretical bounds
on the size of the prime $p$ required to obtain a desired level of certainty
in computing $g_i(X,\CC^8)$ using random choices with computations
over $\ZZ_p$.  

\begin{corollary}
With the setup described above, 
$$
P(g_i(X,\CC^8)_{\ZZ_p}^{\rm rand}=g_i(X,\CC^8))\geq \left(\frac{p-1}{p} \right)^\nu \left(1-\frac{
g_i(X,\CC^8) + \hbox{\rm 
317,987,389,440,000} }{p} \right)
$$where $\nu$ as in Theorem~\ref{thm:ProbPauerLucky} 
which is bounded above by ${ g_i(X,\CC^8)+9 \choose 9}$.
\label{Cor:Prob_Alts}
\end{corollary}\begin{proof}
Using the notation of Corollary \ref{Corr:ProbBound_Affine_OverFiniteFeild},
we have $n = 8$, $r = 15$, $D_{\min}=2$, and $D_{\max}=7$. 
Using Bez\'out's theorem, we obtain the bound $\deg(V)\leq \hbox{7,620,480,000}$. 
\end{proof}

For example, when $i = 0$, $I_0(\Theta,\Lambda,\mu)$ is generated 
by eight general linear combinations of $f$
along with the polynomial $1-(\mu \cdot f)\cdot T$.
The results of \cite{ProductDecomp} can be extended to this situation
to yield $g_0(X,\CC^8) \leq \hbox{18,700}$. 
For $\omega={18709 \choose 9}$, this gives $$
P(g_i(X,\CC^8)_{\ZZ_p}^{\rm rand}=g_i(X,\CC^8))\geq \left(\frac{p-1}{p} \right)^{\omega} \left(1-\frac{\hbox{317,987,389,458,700}}{p} \right).
$$
For example, if we take $p>2^{116}$, we have $P(g_i(X,\CC^8)_{\ZZ_p}^{\rm rand}=g_i(X,\CC^8))\geq 0.99$
Although this bound is pessimistic, it provides
an {\em a priori} lower bound on the probability.
The following compares the probabilistic 
bound for $g_6(X,\CC^8)$
with what is actually achieved in experiments.

\subsection{Practical success rate for $g_{6}(X,\CC^8)$}\label{subsec:PracticalSuccess_g6}

To compare the theoretical lower bounds 
with the practical success rate, we consider 
the simplest nontrivial case, namely computing $g_6(X,\CC^8)$.  
B\'ezout's Theorem together with 
$X$ described in \eqref{eq:AltBaseLocus}
provides $g_6(X,\CC^8)\leq 49-2 = 47$
while the actual value of $43$
can be verified using methods described in Section~\ref{sec:HilbertFunctions}.
Using $g_{6}(X,\CC^8)\leq 47$ with Corollary \ref{Cor:Prob_Alts}, we have
$$
P(g_i(X,\CC^8)_{\ZZ_p}^{\rm rand}=g_i(X,\CC^8))\geq \left(\frac{p-1}{p} \right)^{\hbox{7,575,968,400}} \left(1-\frac{\hbox{317,987,389,440,047}}{p} \right).
$$ 
Therefore, the theoretical results show that $p>2^{55}$ 
will yield a probability of success over~$0.99$.
For this problem, note that the size of $p$ is dominated by
the second term of the product since
$$\log_2\left(100\cdot\hbox{317,987,389,440,047}\right)\geq 54.8
\hbox{\,\,\,\,\,while\,\,\,\,\,}
\left(\frac{2^{55}-1}{2^{55}}\right)^{\hbox{7,575,968,400}}
\geq 
0.9999997897.$$

To check the practical success rate for various smaller primes $p$, 
we performed 10,000 random trials for selected primes $p$.
Table~\ref{table:g6AltsSuccesRate} lists the number of successes,
i.e., when $g_6(X,\CC^8)^{\rm rand}_{\ZZ_p} = 43$.
In particular, this experiment, as expected, shows that
the {\em a priori} theoretical bounds are quite pessimistic.
For example, for all the selected primes $p\geq 2^{11}-9=2039$, 
we already obtained a success rate over $99\%$ in 10,000 random trials.

\begin{table}[h!]
\centering
\resizebox{0.4\linewidth}{!}{
\begin{tabular}{@{} *2c @{}}
\toprule 
 \multicolumn{1}{c}{\color{Ftitle}Prime $p$} &   {\color{Ftitle}Successes in 10,000 Trials}\\ 
 \midrule 
   \multicolumn{1}{l}{$2^6-5$} & ~8,202  \\
  \multicolumn{1}{l}{$2^7-15$} & ~9,057  \\
 \multicolumn{1}{l}{$2^8-5$} &~9,592  \\
 \multicolumn{1}{l}{$2^{11}-9$} &~9,941  \\ 
 \multicolumn{1}{l}{$2^{13}-1$}& ~9,986\\
  \multicolumn{1}{l}{$2^{14}-3$}& ~9,994\\
 \multicolumn{1}{l}{$2^{15}-19$}& ~9,998\\   
 \multicolumn{1}{l}{$2^{16}-17$}& ~9,999\\
 \multicolumn{1}{l}{$2^{17}-1$}& 10,000\\
 \multicolumn{1}{l}{$2^{18}-5$}& 10,000\\
 \multicolumn{1}{l}{$2^{19}-1$}& 10,000\\
\bottomrule
 \end{tabular}}\vspace{1mm}
\caption{Number of successes, i.e.,
when $g_6(X,\CC^8)_{\ZZ_p}^{\rm rand} = 43$,
for 10,000 trials
using various primes $p$.\label{table:g6AltsSuccesRate}}
 \end{table}

\subsection{Larger computations} 

As mentioned in Section~\ref{subSection:AltsIntro},
$g_i(X,\CC^8)$ is an upper bound on the degree
of the set of four-bar linkages whose coupler curve interpolates
$9-i$ general points.  
The actual degree of these sets were first reported
in \cite[Table~1]{Alts1} which were computed using homotopy continuation.
For convenience, we lists these degrees in Table~\ref{table:Degrees}.

\begin{table}[h!]
\centering
\begin{tabular}{c|cccccccc}
{\color{Ftitle}$i$} & 7 & 6 & 5 & 4 & 3 & 2 & 1 & 0 \\
\hline
{\color{Ftitle}degree} & 7 & 43 & 234 & 1108 & 3832 & 8716 & 10858 & 8652\\   
 \end{tabular}\vspace{1mm}
\caption{Degree of the 
set of four-bar linkages whose coupler curve interpolates
$9-i$ general points for $i = 0,\dots,7$.}\label{table:Degrees}
\end{table}

Trivially, $g_7(X,\CC^8) = 7$.
For $i = 6$, the results of Section~\ref{subsec:PracticalSuccess_g6}
show that $g_6(X,\CC^8) = 43$ is also a sharp upper bound.  
Hence, we consider the sharpness for $i = 0,\dots,5$
by computing $g_i(X,\CC^8)_{\ZZ_p}^{\rm rand}$
for various primes $p$ where the parameters
are selected uniformly at random in~$\ZZ_p$.
The results of these computations,
which were performed using a 2.5 GHz Intel Xeon E5-2680~v3 processor with 256 GB of memory,
are summarized in Table~\ref{table:Modular}.
The ones marked with ``$-$'' failed to finish in 10 days,
while the ones in {\bf bold} agree with the 
degrees reported in Table~\ref{table:Degrees}.
In particular, these results
suggest that $g_i(X,\CC^8)$ is indeed a sharp
upper bound for all $i = 0,\dots,7$.
We note that Table~\ref{table:Modular} shows a 
considerable increase in the computational time needed for the two largest
primes in our experiment compared with the other selected primes.  

To provide additional validation of $g_0(X,\CC^8) = 8652$,
we performed 10 additional trails of
$g_i(X,\CC^8)_{\ZZ_p}^{\rm rand}$ for $p = 2^{23}+9$
as well as using randomly selected floating-point parameters 
via homotopy continuation in {\tt Bertini}.
All of these experiments also yielded $8652$.

Table~\ref{table:HF} lists the purported 
Hilbert functions in $\CC[\mathbf{x},T]$.
These are in agreement with numerical computations
based on using solutions computed by
homotopy continuation in {\tt Bertini}.

\begin{table}[ht!]
\centering
\begin{tabular}{c|cccccc}
{\color{Ftitle}$p\setminus i$} & 5 & 4 & 3 & 2 & 1 & 0 \\
\hline
\multirow{2}{*}{$2^{1}+1$} & 176 & 849 & 3352 & 6192 & 8756 & 8492 \\ 
& 0.16s & 4.66s & 2.50m & 0.52h & 4.69h & 0.84d\\
\hline
\multirow{2}{*}{$2^{2}+1$} & 189 & 975 & 3594 & 6803 & 10693 & 8586 \\
& 0.23s & 6.99s & 3.36m & 0.79h & 7.01h & 1.20d\\
\hline
\multirow{2}{*}{$2^{3}+3$} & 233 & 1003 & 3566 & 7673 & 10651 & 8644 \\
& 0.27s & 6.26s & 3.94m & 1.24h & 7.44h & 3.45d\\
\hline
\multirow{2}{*}{$2^{4}+1$} & {\bf 234} & 1059 & 3766 & 8635 & 10739 & 8646 \\
& 0.23s & 7.94s & 4.46m & 1.20h & 7.99h & 1.82d\\
\hline
\multirow{2}{*}{$2^{7}+3$} & {\bf 234} & 1100 & 3812 & {\bf 8716} & {\bf 10858} & {\bf 8652} \\
& 0.29s & 8.49s & 4.38m & 1.42h & 8.55h & 2.69d\\
\hline
\multirow{2}{*}{$2^{11}+5$} & {\bf 234} & 1107 & {\bf 3832} & {\bf 8716} & {\bf 10858} & {\bf 8652} \\
& 0.34s & 7.53s & 4.64m & 1.41h & 9.68h & 3.11d\\
\hline
\multirow{2}{*}{$2^{15}+3$} & {\bf 234} & {\bf 1108} & {\bf 3832} & {\bf 8716} & {\bf 10858} & {\bf 8652} \\
& 0.23s & 8.45s & 4.67m & 1.41h & 8.52h & 2.54d\\
\hline
\multirow{2}{*}{$2^{19}+21$} & {\bf 234} & {\bf 1108} & {\bf 3832} & {\bf 8716} & {\bf 10858} & {\bf 8652} \\
& 0.24s & 7.85s & 4.83m & 1.47h & 9.36h & 3.03d\\
\hline
\multirow{2}{*}{$2^{23}+9$} & {\bf 234} & {\bf 1108} & {\bf 3832} & {\bf 8716} & {\bf 10858} & {\bf 8652} \\
& 0.24s & 7.96s & 4.83m & 1.56h & 9.11h & 2.91d\\
\hline
\multirow{2}{*}{$2^{27}+29$} & {\bf 234} & {\bf 1108} & {\bf 3832} & {\bf 8716} & {\bf 10858} & \multirow{2}{*}{$-$} \\
& 0.81s & 87.63s & 95.23m & 36.76h & 250.31h& \\
\hline
\multirow{2}{*}{$2^{31}+11$} & {\bf 234} & {\bf 1108} & {\bf 3832} & \multirow{2}{*}{$-$} & \multirow{2}{*}{$-$} & \multirow{2}{*}{$-$} \\
& 3.08s & 395.85s & 1385.45m & & & \\
 \end{tabular}\vspace{1mm}
\caption{Values of $g_i(X,\CC^8)_{\ZZ_p}^{\rm rand}$ 
for $i = 0,\dots,5$ and various primes $p$
for randomly selected values of the parameters
along with timings listed in either seconds~(s), minutes (m), hours (h), or days (d).
}\label{table:Modular}
\end{table}

\begin{table}[h!]
\centering
\begin{tabular}{c|ccccccccccc}
{\color{Ftitle}$i\setminus d$} & 0 & 1 & 2 & 3 & 4 & 5 & 6 & 7 & 8 & 9 & 10 \\
\hline
7 & 1 & 3 & 6 & 7 & 7 & 7 & 7 & 7 & 7 & 7 & 7\\
\hline
6 & 1 & 4 & 10 & 20 & 35 & 43 & 43 & 43 & 43 & 43 & 43\\
\hline
5 & 1 & 5 & 15 & 35 & 70 & 126 & 209 & 234 & 234 & 234 & 234\\
\hline
4 & 1 & 6 & 21 & 56 & 126 & 252 & 460 & 777 & 1108 & 1108 & 1108 \\
\hline
3 & 1 & 7 & 28 & 84 & 210 & 462 & 920 & 1686 & 2876 & 3832 & 3832 \\
\hline
2 & 1 & 8 & 36 & 120 & 330 & 791 & 1704 & 3362 & 6154 & 8716 & 8716 \\
\hline
1 & 1 & 9 & 45 & 165 & 495 & 1285 & 2981 & 6307 & 10858 & 10858 & 10858\\
\hline
0 & 1 & 10 & 55 & 220 & 715 & 1999 & 4971 & 8652 & 8652 & 8652 & 8652\\

 \end{tabular}\vspace{1mm}
\caption{Comparison of Hilbert function values for $i = 0,\dots,7$}\label{table:HF}
\end{table}

\section{Conclusion}\label{Section:Conclusion}
For a polynomial system $f$ with $X = \Var(f)\subset \CC^n$ and a variety $Y \subset \CC^n$ containing $X$, 
we considered the computational problem of computing $g_i(X,Y)$ 
for $i=0,\dots, \dim(Y)-1$ as defined in \eqref{eq:giXY}. 
This question was primarily motivated by its application to providing a sharp
upper bound to Alt's problem, namely
$g_0(X,\CC^8)$ where $X = \Var(f_1,\dots,f_{15})$ as formulated
in Section~\ref{subSection:AltsIntro}.
More generally, the values $g_i(X,Y)$ also arise in the computation 
of the volumes of Newton-Okounkov bodies and for the computation of 
characteristic classes such as Chern and Segre classes
as described in Section~\ref{subsec:ProjDegCharClass}.
In particular, Section~\ref{sec:ProbRandomIsGeneral}
provides an analysis of a probabilistic saturation technique 
for computing $g_i(X,Y)$ using Gr\"obner basis
computations over finite fields.

In the context of Alt's problem, the numbers $g_i(X,\CC^8)$ for $i = 0,\dots,7$ 
provide an upper bound on the degree of the set of four-bar linkages
whose coupler curve interpolates $9-i$ general points, which is
a variety of dimension~$i$.  In fact, the solution to Alt's original problem 
from 1923 on the number of distinct four-bar linkages that interpolate $9$ general
points is bounded above by $g_0(X,\CC^8)/2$ and the number of distinct
coupler curves that interpolate~$9$ general points 
is bounded above by $g_0(X,\CC^8)/6$.
In 1992, homotopy continuation was used in \cite{WMS} 
to show the actual numbers were 4326 and 1442, respectively.
However, due to the nature of the computation, 
they could not be sure that no additional solutions could exist. 
In fact, it is reported in \cite{WMS} that 
one of the solutions was missed by their homotopy continuation solver, 
but was reconstructed using Roberts'
cognate formula.  Therefore,
we posit that verifying $g_0(X,\CC^8)=8652$ 
theoretically completes Alt's problem
by showing that there could be no additional solutions.
To that end, we analyzed a probabilistic saturation technique 
(see Theorem \ref{thm:Computeg_i} and Corollary \ref{Corr:ProbBound_Affine_OverFiniteFeild}) 
using Gr\"obner basis computations over various 
finite fields to 
show $g_0(X,\CC^8)^{\rm rand}_{\ZZ_p}=8652$ for
various primes $p$,
a result that is also confirmed numerically using~homotopy~continuation.

As shown in Section~\ref{subSec:Prob}, 
we were able to derive {\em a priori} theoretical bounds 
from Corollary~\ref{Corr:ProbBound_Affine_OverFiniteFeild} 
for the probabilistic computation of 
$g_0(X,\CC^8)=8652$.  This bound requires
the prime characteristic of our finite field to be $p>2^{116}$ 
in order to guarantee a probability of success greater than $0.99$. 
Unfortunately, our computational resources
together with {\tt Magma} do not currently permit
computing $g_0(X,\CC^8)_{\ZZ_p}^{\rm rand}$
for $p> 2^{116}$.
In fact, Table~\ref{table:Modular} shows a drastic increase
in computational time between $p \leq 2^{23}+9$ and $p\geq 2^{27}+29$.

On the other hand, Section~\ref{subsec:PracticalSuccess_g6}
shows the pessimism of the {\em a priori} 
theoretical bounds in that much higher
success rates can be obtained in practice when computing over finite fields with much smaller characteristic.  In particular, 
for the probabilistic computation of $g_{6}(X,\CC^8)=43$, 
the bound of Corollary~\ref{Corr:ProbBound_Affine_OverFiniteFeild} suggests 
one would need to take a finite field with prime characteristic larger than $2^{55}$ to obtain a success rate greater than $99\%$.
Experimentally, 
we found that we achieved this success rate on 10,000 random trials
with all selected primes $p\geq 2^{11}-9=2039$. 
Further, even with the prime $p=2^7-15=113$,
we achieved a success rate of over $90\%$ in our experiment.
Naturally, this leads us to conclude that the bound of 
Corollary~\ref{Corr:ProbBound_Affine_OverFiniteFeild} 
of $2^{116}$ for computing~$g_0(X,\CC^8)$ 
with a success rate greater than $99\%$
is also overly pessimistic.
Coupled with numerical homotopy continuation computations,
we have confidence that the probabilistic computational 
results indeed correctly yield \mbox{$g_0(X,\CC^8)=8652$}.

Finally, we remark that the probabilistic saturation technique has 
permitted computations related to Alt's problem 
to terminate using symbolic methods.
For example, our probabilistic saturation computations 
related to $i = 6$
completed on average in under $0.1$ seconds for our experiments
using both {\tt Magma} and {\tt Macaulay2} \cite{M2} for the primes listed in Table~\ref{table:g6AltsSuccesRate}.
A similar problem was attempted
in \cite[\S 5.4]{Comparison} using primary decomposition over finite fields 
via {\tt Singular}~\cite{DGPS}, a computation which failed to terminate in 24 hours.  In fact, to generate the data for 
Table \ref{table:g6AltsSuccesRate},
the 10,000 random trials for each prime were completed 
in under 10 minutes.  This drastic increase in speed 
allowed us to successful 
complete random trials for computing $g_i(X,\CC^8)$ 
associated with Alt's problem for $i = 0,\dots,6$
for various primes as reported in Section~\ref{sec:Computations}.
In particular, since the values
of $g_i(X,\CC^8)$ match the results
in Table~\ref{table:Degrees}, we have utilized
symbolic methods to confirm 
numerical homotopy continuation computations in \cite{Alts1,WMS}.

\section*{Acknowledgments}

The authors would like to thank the 
Fall 2018 semester program on {\em Nonlinear Algebra}
at the Institute for Computational and Experimental Research in Mathematics (ICERM) where this project was started.
JDH was supported in part by NSF CCF 1812746.

\bigskip

\footnotesize

\noindent {\bf Jonathan D. Hauenstein}. Department of Applied and Computational Mathematics 
and Statistics, University of Notre Dame, Notre Dame, IN, USA.
(\email{hauenstein@nd.edu}, \url{www.nd.edu/~jhauenst}).

\bigskip

\noindent {\bf Martin Helmer}.  Mathematical Sciences Institute, The Australian National University, Canberra, ACT, Australia. (\email{martin.helmer@anu.edu.au}, \url{www.maths.anu.edu.au/people/academics/martin-helmer}).


\bigskip 
\scriptsize
\begin{thebibliography}{plain}

\bibitem{ZeroDim} J.~Abbott, M.~Kreuzer, L.~Robbiano. Computing zero-dimensional schemes. {\em J. Symb. Comput.}, 39(1), 31--49, 2005.

\bibitem{Aluffi} P.~Aluffi. Computing characteristic classes of projective schemes. {\em J. Symb. Comput.}, 35(1), 3--19, 2003.

\bibitem{Alt} H.~Alt. {\"U}ber die {E}rzeugung gegebener ebener {K}urven mit {H}ilfe des {G}elenkvierecks. {\em ZAMM}, 3(1), 
\mbox{13--19},~1923.

\bibitem{ModGB} E.A.~Arnold. Modular algorithms for computing Gr\"obner bases. {\em J. Symb. Comput.}, 35(4), 403--419, 2003.

\bibitem{CCPC} A. Baskar, S. Bandyopadhyay.
An algorithm to compute the finite roots of large systems of polynomial equations arising in kinematic synthesis.
{\em Mech. Mach. Theory}, 133, 493--513, 2019.

\bibitem{Comparison}
D.J. Bates, W. Decker, J.D. Hauenstein, C. Peterson, G. Pfister, F.-O. Schreyer, A.J. Sommese, C.W. Wampler.
Comparison of probabilistic algorithms for analyzing the components of an affine algebraic variety.
{\em Appl. Math. Comput.}, 231, 619--633, 2014. 

\bibitem{Bertini} D.J.~Bates, J.D.~Hauenstein, A.J.~Sommese, C.W. Wampler.
Bertini: Software for Numerical Algebraic Geometry.
Available at \url{bertini.nd.edu}.

\bibitem{BertiniBook} D.J.~Bates, J.D.~Hauenstein, A.J.~Sommese, C.W. Wampler.
{\em Numerically solving polynomial systems with {B}ertini},
{Society for Industrial and Applied Mathematics (SIAM),
              Philadelphia, PA}, 2013.

\bibitem{BKK} D.N.~Bernstein, A.G~Ku\v{s}nirenko, A.G.~Hovanski\u{\i}.
 Newton polyhedra, {\em Uspehi Mat. Nauk}, {31}, \mbox{201--202},~1976.

 \bibitem{alphaTheory}
L. Blum, F. Cucker, M. Shub, S. Smale.
{\em Complexity and real computation}, 
Springer-Verlag, New York,~1998.

\bibitem{Magma}
W. Bosma, J. Cannon, C. Playoust. 
The Magma algebra system. I. The user language.
{\em J. Symb. Comput.}, 24, 235--265, 1997.

\bibitem{Alts1}D.A.~Brake, J.D.~Hauenstein, A.P.~Murray, D.H.~Myszka, C.W.~Wampler.
The complete solution of Alt-Burmester synthesis problems for four-bar linkages. {\em J. Mech. Robotics}, 8(4), 041018, 2016.

\bibitem{DGPS}
W. Decker, G.-M. Greuel, G. Pfister, M. Sch{\"o}nemann.
\newblock {\sc Singular} {4-1-2} --- {A} computer algebra system for polynomial computations.
\newblock {http://www.singular.uni-kl.de} (2019).

\bibitem{FGLM} J. Faug\'ere, P. Gianni, D. Lazard, T. Mora. Efficient computation of zero-dimensional Gr\"obner bases by change of ordering. 
{\em J. Symb. Comput.}, 16(4), 329–-344, 1993.

\bibitem{Fulton}W.~Fulton. {\em Intersection theory}, 2nd ed., Ergebnisse der Mathematik und ihrer Grenzgebiete. 3. Folge. A Series of Modern Surveys in Mathematics, vol. 2, Springer-Verlag, Berlin, 1998.

\bibitem{M2} D.~Grayson, M.~Stillman.
{\em Macaulay2, a software system for research in algebraic geometry}, 
available at {\tt www.math.uiuc.edu/Macaulay2/}.

\bibitem{NumHilbert} Z.A.~Griffin, J.D.~Hauenstein, C.~Peterson, A.J.~Sommese. Numerical computation of the Hilbert function of a zero-scheme. {\em Springer Proceedings in Mathematics \& Statistics}, 76, 235--250, 2014.

\bibitem{AGI}U.~G{\"o}rtz, T.~Wedhorn. {\em Algebraic Geometry: Part I: Schemes}. Springer Science \& Business Media, 2010.

\bibitem{HH18}C.~Harris, M.~Helmer. Segre class computation and practical applications. To appear in {\em Math. Comput.}

\bibitem{Harris}J.~Harris. {\em Algebraic geometry: a first course}. Springer Science \& Business Media, 2013.

\bibitem{MhomTrace}
J.D. Hauenstein, J.I. Rodriguez.
Multiprojective witness sets and a trace test.
To appear in {\em Adv. Geom.}

\bibitem{Regeneration} 
J.D. Hauenstein, A.J. Sommese, C.W. Wampler.
Regeneration homotopies for solving systems of polynomials.
{\em Math. Comp.}, 80, 345--377, 2011.

\bibitem{alphaCertified}
J.D. Hauenstein, F. Sottile.
Algorithm 921: alpha{C}ertified: {C}ertifying solutions to polynomial systems.
{\em ACM Trans. Math. Softw.}, 38(4), 28, 2012. 

\bibitem{H16} M.~Helmer. Algorithms to compute the topological Euler characteristic, Chern-Schwartz-MacPherson class and Segre class of projective varieties. {\em J. Symb. Comput.}, 73, 120--138, 2016.

\bibitem{H17}M.~Helmer. A direct algorithm to compute the topological Euler characteristic and Chern-Schwartz-MacPherson class of projective complete intersection varieties. {\em Theor. Comput. Sci.}, 681, \mbox{54--74},~2017.

\bibitem{KK12} K. Kaveh, A.G. Khovanskii.
Newton-{O}kounkov bodies, semigroups of integral points, graded algebras and intersection theory.
{\em Ann. of Math. (2)}, 176(2), 925--978, 2012.

\bibitem{LM09} R. Lazarsfeld, M. Musta\c{t}\u{a}.
Convex bodies associated to linear series.
{\em Ann. Sci. \'{E}c. Norm. Sup\'{e}r. (4)}, 42(5), 783--865, 2009.

\bibitem{BeyondPolyhedral} A.~Leykin, J.~Yu.
Beyond polyhedral homotopies.  {\em J. Symb. Comput.},
91, 173--190, 2019.

\bibitem{MB12} M.A.~Marco-Buzun{\'a}riz. A polynomial generalization of the Euler characteristic for algebraic sets. {\em Journal of Singularities}, 4, 114--130, 2012.

\bibitem{CoeffParam} A.P.~Morgan, A.J.~Sommese.
Coefficient-parameter polynomial continuation.
{\em Appl. Math. Comput.}, 29(2), 123--160, 1989.

\bibitem{ProductDecomp} A.P.~Morgan, A.J.~Sommese, C.W.~Wampler.
A product-decomposition bound for Bezout numbers.
{\em SIAM J. Num. Anal.}, 32(4), 1308--1325, 1995.

\bibitem{Okunkov96} A.~Okounkov. Brunn-{M}inkowski inequality for multiplicities.  {\em Invent. Math.}, 125(3), 405--411, 1996.

\bibitem{Pauer}F.~Pauer. On lucky ideals for Gr\"obner basis computations. {\em J. Symb. Comput.}, 14, 471--482, 1992.

\bibitem{FRG}
M.M. Plecnik, R.S. Fearing.
Finding only finite roots to large kinematic synthesis systems.
{\em J. Mech. Robot.}, 9(2), 021005, 2017.

\bibitem{R} S.~Roberts.  On three-bar motion in plane space.
{\em Proc. London Mathematical Society, III}, 
286--319, 1875.

\bibitem{RF63} B.~Roth, F.~Freudenstein.  Synthesis of path-generating mechanisms by numerical methods. {\em ASME Journal of Engineering for Industry, Series B}. 85(3):298--306, 1963.

\bibitem{Schwartz}J.T.~Schwartz. Fast probabilistic algorithms for verification of polynomial identities. {\em J. ACM}, 27(4), 701--717, 1980.

\bibitem{ConstrainedHomotopy}
H. Tari, H.-J. Su, T.-Y. Li.
A constrained homotopy technique for excluding unwanted solutions from polynomial equations arising in kinematics problems.
{\em Mech. Mach. Theory}, 45(6), 898--910, 2010.

\bibitem{GBTrace}C.~Traverso. Gr\"obner trace algorithms. In {\em 
International Symposium on Symbolic and Algebraic Computation}. Springer, Berlin, 1988.

\bibitem{WMS} C.W.~Wampler, A.P.~Morgan, A.J.~Sommese.
Complete solution of the nine-point path synthesis problem for four-bar linkages.  {\em ASME J. Mech. Des.}. 114(1), 153--161, 1992.

\bibitem{Zippel}R.~Zippel. Probabilistic algorithms for sparse polynomials.
{\em LNCS}, 72, 216--226, 1979.

\end{thebibliography}
\end{document}